\documentclass[11pt]{article}
\usepackage{CJK}
\usepackage{amsmath,amsfonts,mathrsfs,amssymb}
\usepackage{indentfirst}
\usepackage{amsthm}
\numberwithin{equation}{section}
\usepackage{color}

\usepackage{hyperref,amssymb,cite}
\usepackage{cite}
\usepackage[misc,geometry]{ifsym}  
\hypersetup{hypertex=true,
	colorlinks=true,
	linkcolor=blue,
	anchorcolor=blue,
	citecolor=blue}

\theoremstyle{plain}
\newtheorem{theorem}{Theorem}[section]

\newtheorem{lemma}[theorem]{Lemma}

\theoremstyle{definition}

\numberwithin{equation}{section}

\newcommand{\HT}{\CJKfamily{hei}}

\newcommand{\R}{\mathbb{R}}

\def\be{\begin{equation}}
\def\ee{\end{equation}}
\def\bea{\begin{eqnarray}}
\def\eea{\end{eqnarray}}
\def\bes{\begin{eqnarray*}}
\def\ees{\end{eqnarray*}}

\def\y{\begin{eqnarray*}}
\def\ey{\end{eqnarray*}}

\begin{document}
\title{\HT {Singular Trudinger-Moser inequality involving $L^{p}$ norm in bounded domain}}
\author{\small {Kaiwen Guo$^a$, Yanjun Liu$^b$ }\\
\small $^a$Chern Institute of Mathematics, Nankai University, Tianjin  300071,  P. R. China\\
\small $^b$School of Mathematical Sciences, Chongqing Normal University,\\
\small Chongqing 401331, P. R. China}
\date{}
\maketitle

\date{}
\maketitle
\footnote[0]{This work is supported by the National Natural Science Foundation of China(No. 12201089, No. 12371051), the Natural Science Foundation Project of Chongqing(No. CSTB2022NSCQ-MSX0226), the Science and Technology Research Program of Chongqing Municipal Education Commission(No. KJQN202200513), the Innovation projects for studying abroad and returning to China(No. cx2023097) and Chongqing Normal University Foundation(No. 21XLB039).}\\
\footnote[0]{~~~~~~Kaiwen Guo}
\footnote[0]{~~~~~~gkw17853142261@163.com }
\footnote[0]{\Letter ~ Yanjun Liu}
\footnote[0]{~~~~~~liuyj@mail.nankai.edu.cn }

\noindent{\small
{\bf Abstract:} In this paper, we use the methods of blow-up analysis and capacity estimate to derive singular Trudinger-Moser inequality involving $N$-Finsler-Laplacian and $L^{p}$ norm in bounded domain, precisely, for any $p>1$, $0\leq\gamma<\gamma_{1}:=\inf\limits_{u\in W_{0}^{1,N}(\Omega)\backslash\{0\}}\frac{\int_{\Omega}F(\nabla u)^{N}\;\mathrm{d}x}{\|u\|_{p}^{N}}$ and $0\leq\beta<N$, we  have
\begin{align}
    \sup_{u\in W_{0}^{1,N}(\Omega),\;\int_{\Omega}F(\nabla u)^{N}\;\mathrm{d}x-\gamma\| u\|_{p}^{N}\leq1}\int_{\Omega}\frac{e^{\lambda_{N}(1-\frac{\beta}{N})\lvert u\rvert^{\frac{N}{N-1}}}}{F^{o}(x)^{\beta}}\;\mathrm{d}x<+\infty\notag
    \end{align}
where $\lambda_{N}=N^{\frac{N}{N-1}} \kappa_{N}^{\frac{1}{N-1}}$ and $\kappa_{N}$ is the volume of unit Wulff ball. Moreover, the extremal functions for the inequality are also obtained. When $F=\lvert\cdot\rvert$ and $p=N$, we can obtain the singular version of Tintarev type inequality, namely, for any $0\leq\alpha<\alpha_{1}(\Omega):=\inf\limits_{u\in W_{0}^{1,N}(\Omega)\backslash \{0\}}\frac{\int_{\Omega}\lvert\nabla u\rvert^{N}\;\mathrm{d}x}{\|u\|_{N}^{N}}$ and $0\leq\beta<N$, there holds
$$
\sup_{u\in W_{0}^{1,N}(\Omega),\;\int_{\Omega}\lvert\nabla u\rvert^{N}\;\mathrm{d}x-\alpha\|u\|_{N}^{N}\leq1}\int_{\Omega}\frac{e^{\alpha_{N}(1-\frac{\beta}{N})\lvert u\rvert^{\frac{N}{N-1}}}}{\lvert x\rvert^{\beta}}\;\mathrm{d}x<+\infty,
$$
where $\alpha_{N}:=N^{\frac{N}{N-1}}\omega_{N}^{\frac{1}{N-1}}$ and $ \omega_{N}$ is the volume of unit ball.


\section{Introduction and main results}\label{section1}
As well as we know, in the limiting case, the Sobolev embeddings are replaced by the Trudinger-Moser inequality. Let $\Omega$ be a bounded domain in $\R^{N}$ where $N\geq2$, Trudinger \cite{25}(see also Yudovi\v{c} \cite{12}) proved $W^{1,N}_{0}(\Omega)$ is embedded in the Orlicz space $L_{\varphi_{\alpha}}(\Omega)$ which is determined by Young function $\varphi_{\alpha}(t)=e^{\alpha|t|^{\frac{N}{N-1}}}-1$ for some positive number $\alpha$. In particular,  Moser \cite{19} obtained the sharp constant  $\alpha_{N}=N^{\frac{N}{N-1}}\omega_{N}^{\frac{1}{N-1}}$ such that
\begin{equation}\label{eq1.1}
\sup_{u\in W_0^{1,N}(\Omega),\;\|\nabla u\|_{N}\leq1}\int_{\Omega}e^{\alpha_{N}|u|^{\frac{N}{N-1}}}dx<+\infty,
\end{equation}
where $\omega_{N}$ is the volume of
unit ball in $\mathbb{R}^{N}$.

One important problem on Trudinger-Moser inequality is whether extremal functions exist or not. The existence of extremal functions of inequality \eqref{eq1.1} was firstly obtained by Carleson-Chang \cite{5} when $\Omega$ is the unit ball, then obtained by Struwe \cite{22} when $\Omega$ is close to the unit ball in the sense of measure, finally obtained by Flucher \cite{11} and Lin \cite{14} when $\Omega$ is a general smooth bounded domain.

There are several other related expansions.

When $\Omega$ contains the origin, Adimurthi-Sandeep \cite{2} generalized the Trudinger-Moser inequality to singular version, namely, for any $0\leq\beta<N$,
\begin{align}\label{eq1.2}
    \sup\limits_{u\in W_{0}^{1,N}(\Omega),\;\|\nabla u\|_{N}\leq 1}
    \int_{\Omega}\frac{e^{\alpha_{N}(1-\frac{\beta}{N})\lvert u\rvert^{\frac{N}{N-1}}}}{\lvert x\rvert^{\beta}}\;\mathrm{d}x<+\infty.
\end{align}
Futhermore, when $\Omega$ is smooth, it was proven in Csato-Nguyen-Roy \cite{7} that the supremum can be attained.

Another version was established by Tintarev \cite{23}, namely,
\begin{align}
    \sup_{u\in W_{0}^{1,2}(\Omega),\;\int_{\Omega}\lvert\nabla u\rvert^{2}\;\mathrm{d}x-\alpha\|u\|_{2}^{2}\leq1}\int_{\Omega}e^{4\pi u^{2}}\;\mathrm{d}x<+\infty\;\;\mathrm{for}\;\;0\leq\alpha<\alpha(\Omega),\notag
\end{align}
where $\alpha(\Omega)$ is the first eigenvalue of Laplacian. Yang \cite{31} obtained extremal functions for the inequality, then Nguyen \cite{20} extended the results of Tintarev and Yang to the higher dimension as follows:
\begin{align}\label{eq1.3}
    \sup_{u\in W_{0}^{1,N}(\Omega),\;\int_{\Omega}\lvert\nabla u\rvert^{N}\;\mathrm{d}x-\alpha\|u\|_{N}^{N}\leq1}\int_{\Omega}e^{\alpha_{N}\lvert u\rvert^{\frac{N}{N-1}}}\;\mathrm{d}x<+\infty\;\;\mathrm{for}\;\;0\leq\alpha<\alpha_{1}(\Omega),
\end{align}
where $\alpha_{1}(\Omega)=\inf\limits_{u\in W_{0}^{1,N}(\Omega)\backslash\{0\}}\frac{\int_{\Omega}\lvert\nabla u\rvert^{N}\;\mathrm{d}x}{\|u\|_{N}^{N}}$, moreover, the supremum can be attained.

In addition, Wang-Xia \cite{27} investigated the Trudinger-Moser inequality involving the anisotropic Dirichlet norm $\left(\int_{\Omega}F(\nabla u)^{N}\;\mathrm{d}x\right)^{\frac{1}{N}}$, precisely,
\begin{align}\label{eq1.4}
    \sup\limits_{u\in W_{0}^{1,N}(\Omega),\;\|F(\nabla u)\|_{N}\leq 1}
    \int_{\Omega}e^{\lambda_{N}\lvert u\rvert^{\frac{N}{N-1}}}\;\mathrm{d}x<+\infty,
\end{align}
where $\lambda_{N}=N^{\frac{N}{N-1}}\kappa_{N}^{\frac{1}{N-1}}$ and $\kappa_{N}$ is as described below. Then Zhou-Zhou \cite{33} obtained the existence of extremal functions for the inequality.

Combining \eqref{eq1.3} and \eqref{eq1.4}, Liu \cite{17} established anisotropic Trudinger-Moser inequality involving $L^{p}$ norm as follows:
\begin{align}\label{eq1.5}
    \sup\limits_{u\in W_{0}^{1,N}(\Omega),\;\int_{\Omega}F(\nabla u)^{N}\;\mathrm{d}x-\gamma\|u\|_{p}^{N}\leq 1}
    \int_{\Omega}e^{\lambda_{N}\lvert u\rvert^{\frac{N}{N-1}}}\;\mathrm{d}x<+\infty\;\;\mathrm{for}\;\;0\leq\gamma<\gamma_{1},
\end{align}
where $\gamma_{1}$ is as described below, and the
existence of extremal functions for the inequality.

Now, combining \eqref{eq1.2} and \eqref{eq1.5}, we investigate anisotropic singular Trudinger-Moser inequality involving $L^{p}$ norm in bounded domain and the existence of extremal functions.

Throughout this paper, let $\Omega$ be a smooth bounded domain containing the origin in $\R^{N}$ with $N\geq2$. Denote
$$
\gamma_{1}=\inf_{u\in W_{0}^{1,N}(\Omega)\backslash\{0\}}\frac{\int_{\Omega}F(\nabla u)^{N}\;\mathrm{d}x}{\|u\|_{p}^{N}}
$$
and
$$
\|u\|_{\gamma,p,F}=\left(\int_{\Omega}F(\nabla u)^{N}\;\mathrm{d}x-\gamma\|u\|_{p}^{N}\right)^{\frac{1}{N}}.
$$

\begin{theorem}\label{th1.1}
For any $0\leq\gamma<\gamma_{1}$, $0\leq\beta<N$ and $p>1$, there holds
\begin{equation}\label{eq1.6}
    \Lambda:=\sup_{u\in W_{0}^{1,N}(\Omega),\;\|u\|_{\gamma,p,F}\leq1}\int_{\Omega}\frac{e^{\lambda_{N}(1-\frac{\beta}{N})\lvert u\rvert^{\frac{N}{N-1}}}}{F^{o}(x)^{\beta}}\;\mathrm{d}x<+\infty.
\end{equation}
\end{theorem}

One corollary of Theorem \ref{th1.1} is the following Tintarev type inequality.
\begin{theorem}\label{th1.2}
For any $0\leq\alpha<\alpha_{1}(\Omega)$, $0\leq\beta<N$, there holds
\begin{equation}\label{eq1.7}
    \sup_{u\in W_{0}^{1,N}(\Omega),\;\int_{\Omega}\lvert\nabla u\rvert^{N}\;\mathrm{d}x-\alpha\|u\|_{N}^{N}\leq1}\int_{\Omega}\frac{e^{\alpha_{N}(1-\frac{\beta}{N})\lvert u\rvert^{\frac{N}{N-1}}}}{\lvert x\rvert^{\beta}}\;\mathrm{d}x<+\infty.
\end{equation}
\end{theorem}

Another corollary of Theorem \ref{th1.1} is the Adimurthi-Druet type inequality.
\begin{theorem}\label{th1.3}
For any $0\leq\gamma<\gamma_{1}$, $0\leq\beta<N$ and $p>1$, there holds
\begin{equation}\label{eq1.8}
    \sup_{u\in W_{0}^{1,N}(\Omega),\;\int_{\Omega}F(\nabla u)^{N}\;\mathrm{d}x\leq1}\int_{\Omega}\frac{e^{\lambda_{N}(1-\frac{\beta}{N})\lvert u\rvert^{\frac{N}{N-1}}(1+\gamma\|u\|_{p}^{N})^{\frac{1}{N-1}}}}{F^{o}(x)^{\beta}}\;\mathrm{d}x<+\infty.
\end{equation}
\end{theorem}
In fact, if $u\in W_{0}^{1,N}(\Omega)$ and $\int_{\Omega}F(\nabla u)^{N}\;\mathrm{d}x\leq1$, we set $v=(1+\gamma\|u\|_{p}^{N})^{\frac{1}{N}}u\in W_{0}^{1,N}(\Omega)$. Thus
\begin{align}
    \|v\|_{\gamma,p,F}^{N}=(1+\gamma\|u\|_{p}^{N})(\int_{\Omega}F(\nabla u)^{N}\;\mathrm{d}x-\gamma\|u\|_{p}^{N})\leq1-\gamma^{2}\|u\|_{p}^{2N}\leq1.\notag
\end{align}
Applying Theorem \ref{th1.1} to function $v$, we can obtain Theorem \ref{th1.3}.

Theorem \ref{th1.3} generalizes the conclusions in Adimurthi-Druet \cite{1}, Lu-Yang \cite{18}, Yang \cite{30}, Zhu \cite{34}, Wang-Miao \cite{28} and Zhou \cite{32}.

Finally, we give the existence of extremal functions of inequality \eqref{eq1.6}.
\begin{theorem}\label{th1.4}
The supremum
$$
\Lambda=\sup_{u\in W_{0}^{1,N}(\Omega),\;\|u\|_{\gamma,p,F}\leq1}\int_{\Omega}\frac{e^{\lambda_{N}(1-\frac{\beta}{N})\lvert u\rvert^{\frac{N}{N-1}}}}{F^{o}(x)^{\beta}}\;\mathrm{d}x
$$
can be attained by $u_{0}\in W_{0}^{1,N}(\Omega)$ with $\|u_{0}\|_{\gamma,p,F}=1$.
\end{theorem}

We organize this paper as follows. In Section \ref{section2}, we give some preliminaries. Meanwhile, under anisotropic Dirichlet norm and $L^{p}$ norm, we establish the Lions type concentration-compactness principle for singular Trudinger-Moser inequality. In Section \ref{section3}, we prove the existence of maximizers for subcritical inequality. In Section \ref{section4}, we analyze the asymptotic behavior of extremal functions in subcritical cases and establish capacity estimate under the assumption that the blow-up phenomena occurs. In Section \ref{section5}, we prove Theorem \ref{th1.1} and Theorem \ref{th1.4} by the results obove and the construction of test functions.

\section{Preliminaries}\label{section2}

In this section, we provide some preliminary information we will need later.

Let $F:\R^{N}\rightarrow[0,+\infty)$ be a convex function of class $C^{2}(\R^{N}\backslash\{0\})$, which is even and positively homogeneous of degree 1, then there holds
$$
F(tx)=\lvert t\rvert F(x)\;\;\mathrm{for}\;\mathrm{any}\;t\in\R,\;x\in\R^{N}.
$$

We further assume $F(x)>0$ for any $x\neq 0$ and $Hess(F^{2})$ is positive definite in $\R^{N}\backslash\{0\}$, which leading $Hess(F^{N})$ is positive definite in $\R^{N}\backslash\{0\}$ by Xie and Gong \cite{29}. There are two constants $0<a\leq b<\infty$ such that $a\lvert x\rvert\leq F(x)\leq b\lvert x\rvert$ for any $x\in\R^{N}$ and a typical example for $F$ is $F(x)=(\sum\limits_{i}\lvert x_{i}\rvert^{q})^{\frac{1}{q}}$ for $q\in(1,+\infty)$.

Considering the minimization problem
$$
\min_{u\in W^{1,N}(\R^{N})}\int_{\R^{N}}F(\nabla u)^{N}\;\mathrm{d}x,
$$
its Euler-Lagrange equation contains an operator of the form
$$
Q_{N}(u):=\sum_{i=1}^{N}\frac{\partial}{\partial x_{i}}(F(\nabla u)^{N-1}F_{x_{i}}(\nabla u)),
$$
which is called as $N$-anisotropic Laplacian or $N$-Finsler Laplacian.

Let $F^{o}$ be the support function of $K:=\{x\in\R^{N}:F(x)\leq 1\}$, which is defined by
$$
F^{o}(x):=\sup_{\xi\in K}\langle x,\xi\rangle,
$$
then $F^{o}:\R^{N}\rightarrow[0,+\infty)$ is also a convex, positively homogeneous function of class $C^{2}(\R^{N}\backslash\{0\})$.

From \cite{3}, $F^{o}$ is dual to $F$ in the sense that
$$
F^{o}(x)=\sup_{\xi\neq 0}\frac{\langle x,\xi\rangle}{F(\xi)}, \;F(x)=\sup_{\xi\neq 0}\frac{\langle x,\xi\rangle}{F^{o}(\xi)}.
$$

Consider a map $\Phi:S^{N-1}\rightarrow\R^{N}$ satisfying $\Phi(\xi)=\nabla F(\xi)$. Its image $\Phi(S^{N-1})$ is a smooth, convex hypersurface in $\R^{N}$, which is known as the Wulff shape (or equilibrium crystal shape) of $F$. As a result, $\Phi(S^{N-1})=\{x\in\R^{N}|F^{o}(x)=1\}$(see \cite{26}, Proposition 1). Denote $W_{r}^{x_{0}}=\{x\in\R^{N}:F^{o}(x-x_{0})\leq r\}$. We call $W_{r}$ as a Wulff ball with radius $r$ and center at the origin and denote $\kappa_{N}$ as the Lebesgue measure of $W_{1}$.

Accordingly, we provide some simple properties of $F$, as a direct consequence of the assumption on $F$, also found in \cite{4,10,26,27}.
\begin{lemma}\label{le2.1}
There are\\
(i) $\lvert F(x)-F(y)\rvert\leq F(x+y)\leq F(x)+F(y)$;\\
(ii) $\frac{1}{C}\leq\lvert\nabla F(x)\rvert\leq C$ and $\frac{1}{C}\leq \lvert\nabla F^{o}(x)\rvert\leq C$ for some $C>0$ and any $x\neq0$;\\
(iii) $\langle x,\nabla F(x)\rangle=F(x), \langle x, \nabla F^{o}(x)\rangle = F^{o}(x)$ for  any $x\neq 0$;\\
(iv) $\nabla F(tx)=sgn(t)\nabla F(x)$ for any $x\neq 0$ and $t\neq 0$;\\
(v) $\int_{\partial W_r}\frac{1}{|\nabla F^{o}(x)|}\;\mathrm{d}\sigma= N\kappa_Nr^{N-1}$;\\
(vi) $F(\nabla F^{o}(x))=1, F^{o}(\nabla F(x))=1$   for any $x\neq 0$;\\
(vii) $\langle F(x),\nabla F^{o}(\nabla F(x))\rangle=x,\langle F^{o}(x),\nabla F(\nabla F^{o}(x))\rangle=x$ for any $x\neq 0$.
\end{lemma}

Next, we give the Lions type concentration-compactness principle in \cite{15} for singular Trudinger-Moser inequality under anisotropic Dirichlet norm and $L^{p}$ norm in bounded domain, which is the extention of Theorem 1.2 in \cite{16}.

\begin{lemma}\label{le2.2}
Let $u\in W_{0}^{1,N}(\Omega)\backslash\{0\}$, $0\leq\gamma<\gamma_{1}$ and $0\leq\beta<N$. Assume $\{u_{k}\}$ is a seqence in $W_{0}^{1,N}(\Omega)$ such that $\|u_{k}\|_{\gamma,p,F}=1$ and $u_{k}\rightharpoonup u$ weakly in $W_{0}^{1,N}(\Omega)$. If
$$
0<q<q_{N}(u):=\frac{1}{\left(1-\|u\|_{\gamma,p,F}^{N}\right)^{\frac{1}{N-1}}},
$$
then
$$
\sup_{k}\int_{\Omega}\frac{e^{q\lambda_{N}(1-\frac{\beta}{N})\lvert u_{k}\rvert^{\frac{N}{N-1}}}}{F^{o}(x)^{\beta}}\;\mathrm{d}x<+\infty.
$$
\end{lemma}
\begin{proof}
    Since $\|u_{k}\|_{\gamma,p,F}=1$ and $u_{k}\rightharpoonup u$ weakly in $W_{0}^{1,N}(\Omega)$, there is
    $$
    \lim_{k\rightarrow+\infty}\int_{\Omega}F(\nabla u_{k})^{N}\;\mathrm{d}x=\lim_{k\rightarrow+\infty}(1+\gamma\|u_{k}\|_{p}^{N})=1+\gamma\|u\|_{p}^{N}.
    $$
    Set
    $$
    v_{k}=\frac{u_{k}}{\left(\int_{\Omega}F(\nabla u_{k})^{N}\;\mathrm{d}x\right)^{\frac{1}{N}}},
    $$
    then
    $$
    \int_{\Omega}F(\nabla v_{k})^{N}\;\mathrm{d}x=1
    $$
    and
    $$
    v_{k}\rightharpoonup v:=\frac{u}{(1+\gamma\|u\|_{p}^{N})^{\frac{1}{N}}}\;\mathrm{weakly}\;\mathrm{in}\;W_{0}^{1,N}(\Omega).$$
    By Theorem 1.2 in \cite{16},
    \begin{align}\label{eq2.1}
        \sup_{k}\int_{\Omega}\frac{e^{t\lambda_{N}(1-\frac{\beta}{N})\lvert v_{k}\rvert^{\frac{N}{N-1}}}}{F^{o}(x)^{\beta}}\;\mathrm{d}x<+\infty
    \end{align}
    for any $0<t<\left(1-\int_{\Omega}F(\nabla v)^{N}\;\mathrm{d}x\right)^{-\frac{1}{N-1}}$. Meanwhile
    \begin{align}
        \lim_{k\rightarrow\infty}q\left(\int_{\Omega}F(\nabla u_{k})^{N}\;\mathrm{d}x\right)^{\frac{1}{N-1}}&=q\left(1+\gamma\|u\|_{p}^{N}\right)^{\frac{1}{N-1}}\notag\\
        &<\left(\frac{1+\gamma\|u\|_{p}^{N}}{1-\|u\|_{\gamma,p,F}^{N}}\right)^{\frac{1}{N-1}}=\left(1-\int_{\Omega}F(\nabla v)^{N}\;\mathrm{d}x\right)^{-\frac{1}{N-1}}.\notag
    \end{align}
    Now take $t>0$ such that
    $$q(\int_{\Omega}F(\nabla u_{k})^{N}\;\mathrm{d}x)^{\frac{1}{N-1}}<t<\left(1-\int_{\Omega}F(\nabla v)^{N}\;\mathrm{d}x\right)^{-\frac{1}{N-1}}
    $$
    for $k$ large enough. Hence,
    \begin{align}
        &\int_{\Omega}\frac{e^{q\lambda_{N}(1-\frac{\beta}{N})\lvert u_{k}\rvert^{\frac{N}{N-1}}}}{F^{o}(x)^{\beta}}\;\mathrm{d}x\notag\\
        =&\int_{\Omega}\frac{e^{q\lambda_{N}(1-\frac{\beta}{N})(\int_{\Omega}F(\nabla u_{k})^{N}\;\mathrm{d}x)^{\frac{1}{N-1}}\lvert v_{k}\rvert^{\frac{N}{N-1}}}}{F^{o}(x)^{\beta}}\;\mathrm{d}x\notag\\
        <&\int_{\Omega}\frac{e^{t\lambda_{N}(1-\frac{\beta}{N})\lvert v_{k}\rvert^{\frac{N}{N-1}}}}{F^{o}(x)^{\beta}}\;\mathrm{d}x\notag
    \end{align}
    for $k$ large enough. So Lemma \ref{le2.2} is followed from \eqref{eq2.1}.
\end{proof}

\section{Maximizers for subcritical
inequalities}\label{section3}

In this section, we prove the existence of extremal functions for subcritical inequalities. Namely,
\begin{lemma}\label{le3.1}
    For any $0<\epsilon<1-\frac{\beta}{N}$, the supremum
    \begin{align}\label{eq3.1}
        \Lambda_{\epsilon}:=\sup_{u\in W_{0}^{1,N}(\Omega),\;\|u\|_{\gamma,p,F}\leq1}\int_{\Omega}\frac{e^{\lambda_{N}(1-\frac{\beta}{N}-\epsilon)\lvert u\rvert^{\frac{N}{N-1}}}}{F^{o}(x)^{\beta}}\;\mathrm{d}x
    \end{align}
    can be attained by $u_{\epsilon}\in W_{0}^{1,N}(\Omega)$ with $\|u_{\epsilon}\|_{\gamma,p,F}=1$. In the distributional sense, $u_{\epsilon}$ satisfies the following equation:
    \begin{align}\label{eq3.2}
    \left\{\begin{array}{l}
    -Q_{N}(u_{\epsilon})-\gamma\|u_{\epsilon}\|_{p}^{N-p}u_{\epsilon}^{p-1}=\lambda_{\epsilon}^{-1}F^{o}(x)^{-\beta}u_{\epsilon}^{\frac{1}{N-1}}e^{\lambda_{N}(1-\frac{\beta}{N}-\epsilon)u_{\epsilon}^{\frac{N}{N-1}}}\;\;\;\;\;in\;\Omega,\\
    \lambda_{\epsilon}=\int_{\Omega}F^{o}(x)^{-\beta}u_{\epsilon}^{\frac{N}{N-1}}e^{\lambda_{N}(1-\frac{\beta}{N}-\epsilon)u_{\epsilon}^{\frac{N}{N-1}}}\;\mathrm{d}x,\\
    u_{\epsilon}\geq0\;\;in\;\Omega.
    \end{array}\right.
    \end{align}
    Moreover,
    \begin{align}\label{eq3.3}
        \liminf_{\epsilon\rightarrow0}\lambda_{\epsilon}>0
    \end{align}
    and
    \begin{align}\label{eq3.4}
        \lim_{\epsilon\rightarrow0}\Lambda_{\epsilon}=\Lambda.
    \end{align}
\end{lemma}
\begin{proof}
    Let $\{u_{k}\}$ be a maximizing sequence for $\Lambda_{\epsilon}$, i.e., $u_{k}\in W_{0}^{1,N}(\Omega)$, $\|u_{k}\|_{\gamma,p,F}\leq1$ and
    $$
    \lim_{k\rightarrow+\infty}\int_{\Omega}\frac{e^{\lambda_{N}(1-\frac{\beta}{N}-\epsilon)\lvert u_{k}\rvert^{\frac{N}{N-1}}}}{F^{o}(x)^{\beta}}\;\mathrm{d}x=\Lambda_{\epsilon}.
    $$
    Since $\|\lvert u_{k}\rvert\|_{\gamma,p,F}\leq\|u_{k}\|_{\gamma,p,F}\leq1$, we can assume $\{u_{k}\}$ is nonnegative.

    Notice that
    \begin{align}
        1\geq\|u_{k}\|_{\gamma,p,F}^{N}=\int_{\Omega}F(\nabla u_{k})^{N}\;\mathrm{d}x-\gamma\|u_{k}\|_{p}^{N}\geq\left(1-\frac{\gamma}{\gamma_{1}}\right)\int_{\Omega}F(\nabla u_{k})^{N}\;\mathrm{d}x\notag
    \end{align}
    and $\gamma<\gamma_{1}$, we have $\{u_{k}\}$ is bounded in $W_{0}^{1,N}(\Omega)$. Up to a subsequence, there exists some $u_{\epsilon}\in W_{0}^{1,N}(\Omega)$ such that
    \begin{align}
        &u_{k}\rightharpoonup u_{\epsilon}\;\mathrm{weakly}\;\mathrm{in}\;W_{0}^{1,N}(\Omega),\notag\\
        &u_{k}\rightarrow u_{\epsilon}\;\mathrm{strongly}\;\mathrm{in}\;L^{i}(\Omega)\;\mathrm{for}\;i\in[1,+\infty),\notag\\
        &u_{k}\rightarrow u_{\epsilon}\;\mathrm{a.e.}\;\mathrm{in}\;\Omega\notag
    \end{align}
    as $k\rightarrow+\infty$. Obviously, $u_{\epsilon}$ is also nonnegative and $\|u_{\epsilon}\|_{\gamma,p,F}\leq1$.

    For any $i\geq1$ and $s>1$, let $s'=\frac{s}{s-1}$, by the H\"{o}lder inequality,
    $$
    \int_{\Omega}\frac{e^{i\lambda_{N}(1-\frac{\beta}{N}-\epsilon)u_{k}^{\frac{N}{N-1}}}}{F^{o}(x)^{i\beta}}\;\mathrm{d}x\leq\left(\int_{\Omega}e^{s'i\lambda_{N}(1-\frac{\beta}{N}-\epsilon)u_{k}^{\frac{N}{N-1}}}\;\mathrm{d}x\right)^{\frac{1}{s'}}\left(\int_{\Omega}\frac{1}{F^{o}(x)^{si\beta}}\;\mathrm{d}x\right)^{\frac{1}{s}}.
    $$
    Then by Theorem 1.1 in \cite{17}, we have $\frac{e^{\lambda_{N}(1-\frac{\beta}{N}-\epsilon)u_{k}^{\frac{N}{N-1}}}}{F^{o}(x)^{\beta}}$ is bounded in $L^{i}(\Omega)$ for any $1\leq i<\frac{1}{1-\epsilon}$. Since
    \begin{align}
    &\frac{\lvert e^{\lambda_{N}(1-\frac{\beta}{N}-\epsilon)u_{k}^{\frac{N}{N-1}}}-e^{\lambda_{N}(1-\frac{\beta}{N}-\epsilon)u_{\epsilon}^{\frac{N}{N-1}}}\rvert}{F^{o}(x)^{\beta}}\notag\\
    \leq&\lambda_{N}(1-\frac{\beta}{N}-\epsilon)\frac{e^{\lambda_{N}(1-\frac{\beta}{N}-\epsilon)u_{k}^{\frac{N}{N-1}}}+e^{\lambda_{N}(1-\frac{\beta}{N}-\epsilon)u_{\epsilon}^{\frac{N}{N-1}}}}{F^{o}(x)^{\beta}}\lvert u_{k}^{\frac{N}{N-1}}-u_{\epsilon}^{\frac{N}{N-1}}\rvert\notag
    \end{align}
    and $u_{k}\rightarrow u_{\epsilon}$ strongly in $L^{i}(\Omega)$ for $i\in[1,+\infty)$, we conclude
    \begin{align}\label{eq3.5}
        \lim_{k\rightarrow+\infty}\int_{\Omega}\frac{e^{\lambda_{N}(1-\frac{\beta}{N}-\epsilon)u_{k}^{\frac{N}{N-1}}}}{F^{o}(x)^{\beta}}\;\mathrm{d}x=\int_{\Omega}\frac{e^{\lambda_{N}(1-\frac{\beta}{N}-\epsilon)u_{\epsilon}^{\frac{N}{N-1}}}}{F^{o}(x)^{\beta}}\;\mathrm{d}x,
    \end{align}
    which implies $u_{\epsilon}$ is a maximizer for $\Lambda_{\epsilon}$.

    Clearly $u_{\epsilon}\not\equiv0$ and we claim $\|u_{\epsilon}\|_{\gamma,p,F}=1$. If otherwise,
    \begin{align}
    \Lambda_{\epsilon}=\int_{\Omega}\frac{e^{\lambda_{N}(1-\frac{\beta}{N}-\epsilon)u_{\epsilon}^{\frac{N}{N-1}}}}{F^{o}(x)^{\beta}}\;\mathrm{d}x<\int_{\Omega}\frac{e^{\lambda_{N}(1-\frac{\beta}{N}-\epsilon)(\frac{u_{\epsilon}}{\|u_{\epsilon}\|_{\gamma,p,F}})^{\frac{N}{N-1}}}}{F^{o}(x)^{\beta}}\;\mathrm{d}x\leq\Lambda_{\epsilon},\notag
    \end{align}
    which is a contradiction.

    Furthermore, $u_{\epsilon}$ satisfies the Euler-Lagrange equation \eqref{eq3.2} in the distributional sense. Applying elliptic regularity theory obtained by Serrin (\cite{21}, Theorem 6 and Theorem 8) and Tolksdorf (\cite{24}, Theorem 1), we obtain $u_{\epsilon}\in C^{0}(\Omega)\cap C_{\mathrm{loc}}^{1}(\Omega\backslash\{0\})$.

    Obviously, $\limsup\limits_{\epsilon\rightarrow0}\Lambda_{\epsilon}\leq\Lambda$. On the other hand, for any $u\in W_{0}^{1,N}(\Omega)$ with $\|u\|_{\gamma,p,F}\leq1$, by Fatou's lemma,
    \begin{align}
        \int_{\Omega}\frac{e^{\lambda_{N}(1-\frac{\beta}{N})\lvert u\rvert^{\frac{N}{N-1}}}}{F^{o}(x)^{\beta}}\;\mathrm{d}x\leq\liminf_{\epsilon\rightarrow0}\int_{\Omega}\frac{e^{\lambda_{N}(1-\frac{\beta}{N}-\epsilon)\lvert u\rvert^{\frac{N}{N-1}}}}{F^{o}(x)^{\beta}}\;\mathrm{d}x\leq\liminf_{\epsilon\rightarrow0}\Lambda_{\epsilon},\notag
    \end{align}
    which implies $\liminf\limits_{\epsilon\rightarrow0}\Lambda_{\epsilon}\geq\Lambda$. Thus $\lim\limits_{\epsilon\rightarrow0}\Lambda_{\epsilon}=\Lambda$.

    Finally, by $e^{t}\leq1+te^{t}$ for any $t\geq0$,
    \begin{align}
        \liminf_{\epsilon\rightarrow0}\lambda_{\epsilon}&\geq\lim_{\epsilon\rightarrow0}\frac{1}{\lambda_{N}(1-\frac{\beta}{N}-\epsilon)}\int_{\Omega}\frac{e^{\lambda_{N}(1-\frac{\beta}{N}-\epsilon)u_{\epsilon}^{\frac{N}{N-1}}}-1}{F^{o}(x)^{\beta}}\;\mathrm{d}x\notag\\
        &=\frac{1}{\lambda_{N}(1-\frac{\beta}{N})}(\Lambda-\int_{\Omega}F^{o}(x)^{-\beta}\;\mathrm{d}x)\notag\\
        &>0.\notag
    \end{align}
\end{proof}

\section{Asymptotic behavior of maximizers for the subcritical inequalities}\label{section4}

In this section, we analyze the blow-up behavior of extremal functions in the subcritical cases and establish capacity estimate under the assumption that the blow-up phenomena occur.
\subsection{Energy concentration phenomenon}

Since $\|u_{\epsilon}\|_{\gamma,p,F}=1$ and $\gamma<\gamma_{1}$, we have $\{u_{\epsilon}\}$ is bounded in $W_{0}^{1,N}(\Omega)$. Up to a subsequence, there exists some $u_{0}\in W_{0}^{1,N}(\Omega)$ such that
\begin{align}
    &u_{\epsilon}\rightharpoonup u_{0}\;\mathrm{weakly}\;\mathrm{in}\;W_{0}^{1,N}(\Omega),\notag\\
    &u_{\epsilon}\rightarrow u_{0}\;\mathrm{strongly}\;\mathrm{in}\;L^{i}(\Omega),\;i\in[1,+\infty),\notag\\
    &u_{\epsilon}\rightarrow u_{0}\;\mathrm{a.e.}\;\mathrm{in}\;\Omega\notag
\end{align}
as $\epsilon\rightarrow0$.

Denote $c_{\epsilon}=\max\limits_{x\in\overline{\Omega}}u_{\epsilon}(x)=u_{\epsilon}(x_{\epsilon})$ for some $x_{\epsilon}\in\overline{\Omega}$. Up to a subsequence, we can assume $x_{\epsilon}\rightarrow x_{0}\in\overline{\Omega}$. If $c_{\epsilon}$ is bounded as $\epsilon\rightarrow0$, then for any $u\in W_{0}^{1,N}(\Omega)$ with $\|u\|_{\gamma,p,F}\leq1$, by the Lebesgue dominated convergence theorem,
\begin{align}\label{eq4.1}
    \Lambda=\lim_{\epsilon\rightarrow0}\int_{\Omega}\frac{e^{\lambda_{N}(1-\frac{\beta}{N}-\epsilon)u_{\epsilon}^{\frac{N}{N-1}}}}{F^{o}(x)^{\beta}}\;\mathrm{d}x=\int_{\Omega}\frac{e^{\lambda_{N}(1-\frac{\beta}{N})u_{0}^{\frac{N}{N-1}}}}{F^{o}(x)^{\beta}}\;\mathrm{d}x
\end{align}
which implies  $u_{0}$ is a maximizer for $\Lambda$ and $\|u_{0}\|_{\gamma,p,F}=1$.

In the following, we assume $c_{\epsilon}\rightarrow+\infty$.
\begin{lemma}\label{le4.1}
    $u_{0}\equiv0$, $x_{0}=0$ and $F(\nabla u_{\epsilon})^{N}\mathrm{d}x\rightharpoonup \delta_{0}$ weakly in the sense of measure as $\epsilon\rightarrow0$, where $\delta_{0}$ denotes the Dirac measure centered at $0$.
\end{lemma}
\begin{proof}
    Suppose $u_{0}\not\equiv0$. By Lemma \ref{le2.2} and $\|u_{\epsilon}\|_{p}^{N-p}u_{\epsilon}^{p-1}$ is bounded in $L^{\frac{p}{p-1}}(\Omega)$, there exists some $i>1$ such that $-Q_{N}(u_{\epsilon})$ is bounded in $L^{i}(\Omega)$. Then by Lemma 2.2 in \cite{33}, $u_{\epsilon}$ is uniformly bounded in $\Omega$, which contradicts $c_{\epsilon}\rightarrow+\infty$ as $\epsilon\rightarrow0$. Hence $u_{0}\equiv0$.

    Suppose $x_{0}\not=0$. Then $F^{o}(x)^{-\beta}e^{\lambda_{N}(1-\frac{\beta}{N}-\epsilon)u_{\epsilon}^{\frac{N}{N-1}}}$ is bounded in $L^{i_{1}}(\Omega\cap W_{\frac{F^{o}(x_{0})}{2}})$ for some $i_{1}>1$. In view of $F^{o}(x)^{-\beta}\leq(\frac{F^{o}(x_{0})}{2})^{-\beta}$ when $F^{o}(x)\geq\frac{F^{o}(x_{0})}{2}$ and Theorem 1.1 in \cite{17}, we have $F^{o}(x)^{-\beta}e^{\lambda_{N}(1-\frac{\beta}{N}-\epsilon)u_{\epsilon}^{\frac{N}{N-1}}}$ is bounded in $L^{i_{2}}(\Omega\backslash W_{\frac{F^{o}(x_{0})}{2}})$ for some $i_{2}>1$. Therefore $-Q_{N}(u_{\epsilon})$ is bounded in $L^{i}(\Omega)$ for some $i>1$ and by Lemma 2.2 in \cite{33}, $u_{\epsilon}$ is uniformly bounded in $\Omega$, which contradicts $c_{\epsilon}\rightarrow+\infty$ as $\epsilon\rightarrow0$. Hence $x_{0}=0$.

    If $F(\nabla u_{\epsilon})^{N}\mathrm{d}x\rightharpoonup\mu\not=\delta_{0}$ weakly in the sense of measure as $\epsilon\rightarrow0$. Notice that $\int_{\Omega}F(\nabla u_{\epsilon})^{N}\;\mathrm{d}x=1+o_{\epsilon}(1)$, there exists $r>0$ small enough and $\theta<1$ such that $W_{r}\subset\Omega$ and
    $$
    \lim_{\epsilon\rightarrow0}\int_{W_{r}}F(\nabla u_{\epsilon})^{N}\;\mathrm{d}x\leq\theta<1.
    $$
    Consider the cut-off function $\phi\in C_{0}^{\infty}(W_{r})$, which is equal to $1$ on $W_{\frac{r}{2}}$, since $u_{\epsilon}\rightarrow0$ in $L^{i}(\Omega)$ for all $i\in[1,+\infty)$, we have
    $$
    \limsup_{\epsilon\rightarrow0}\int_{W_{r}}F(\nabla(\phi u_{\epsilon}))^{N}\;\mathrm{d}x\leq\lim_{\epsilon\rightarrow0}\int_{W_{r}}F(\nabla u_{\epsilon})^{N}\;\mathrm{d}x\leq\theta<1.
    $$
    For sufficiently small $\epsilon>0$, using anisotropic singular Trudinger-Moser inequality (see \cite{35}, Theorem 2.1) to $\phi u_{\epsilon}$ on $W_{r}$, there exists some $i>1$ such that $-Q_{N}(u_{\epsilon})$ is bounded in $L^{i}(W_{\frac{r}{2}})$. By Lemma 2.2 in \cite{33}, $u_{\epsilon}$ is uniformly bounded in $W_{\frac{r}{2}}$, which contradicts $c_{\epsilon}\rightarrow+\infty$ as $\epsilon\rightarrow0$. Hence $F(\nabla u_{\epsilon})^{N}\mathrm{d}x\rightharpoonup \delta_{0}$ weakly in the sense of measure as $\epsilon\rightarrow0$.
\end{proof}

\subsection{Blow-up analysis}

Let
\begin{align}\label{eq4.2}
    r_{\epsilon}^{N}=\lambda_{\epsilon}c_{\epsilon}^{-\frac{N}{N-1}}e^{-\lambda_{N}(1-\frac{\beta}{N}-\epsilon)c_{\epsilon}^{\frac{N}{N-1}}}.
\end{align}
We have
\begin{align}\label{eq4.3}
    r_{\epsilon}\rightarrow0\;\;\mathrm{as}\;\;\epsilon\rightarrow0.
\end{align}

In fact, for any $i\in\R$,
\begin{align}
    c_{\epsilon}^{i}r_{\epsilon}^{N}&=c_{\epsilon}^{i-\frac{N}{N-1}}e^{-\lambda_{N}(1-\frac{\beta}{N}-\epsilon)c_{\epsilon}^{\frac{N}{N-1}}}\int_{\Omega}F^{o}(x)^{-\beta}u_{\epsilon}^{\frac{N}{N-1}}e^{\lambda_{N}(1-\frac{\beta}{N}-\epsilon)u_{\epsilon}^{\frac{N}{N-1}}}\;\mathrm{d}x\notag\\
    &\leq c_{\epsilon}^{i}e^{-\frac{\lambda_{N}}{2}(1-\frac{\beta}{N}-\epsilon)c_{\epsilon}^{\frac{N}{N-1}}}\int_{\Omega}F^{o}(x)^{-\beta}e^{\frac{\lambda_{N}}{2}(1-\frac{\beta}{N}-\epsilon)u_{\epsilon}^{\frac{N}{N-1}}}\;\mathrm{d}x\notag\\
    &\leq Cc_{\epsilon}^{i}e^{-\frac{\lambda_{N}}{2}(1-\frac{\beta}{N}-\epsilon)c_{\epsilon}^{\frac{N}{N-1}}}\rightarrow0,\notag
\end{align}
as $\epsilon\rightarrow0$. Here we use the H\"{o}lder inequality and Theorem 1.1 in \cite{17}.

We now distinguish two cases to proceed.

$\emph{\textbf{Case\;1.}}$ $\frac{F^{o}(x_{\epsilon})^{1-\frac{\beta}{N}}}{r_{\epsilon}}\rightarrow+\infty$.

Define
\begin{align}
    \Psi_{\epsilon}(x)=c_{\epsilon}^{-1}u_{\epsilon}(x_{\epsilon}+r_{\epsilon}F^{o}(x_{\epsilon})^{\frac{\beta}{N}}x)\notag
\end{align}
and
\begin{align}
    \Phi_{\epsilon}(x)=c_{\epsilon}^{\frac{N}{N-1}}(\Psi_{\epsilon}(x)-1)=c_{\epsilon}^{\frac{1}{N-1}}(u_{\epsilon}(x_{\epsilon}+r_{\epsilon}F^{o}(x_{\epsilon})^{\frac{\beta}{N}}x)-c_{\epsilon}).\notag
\end{align}
on $\Omega_{\epsilon}=\{x\in\R^{N}:x_{\epsilon}+r_{\epsilon}F^{o}(x_{\epsilon})^{\frac{\beta}{N}}x\in\Omega\}$.

A straightforward calculation shows  $\Psi_{\epsilon}$ is a weak solution to
\begin{align}\label{eq4.4}
    &-Q_{N}(\Psi_{\epsilon})(x)-\gamma\|u_{\epsilon}\|_{p}^{N-p}c_{\epsilon}^{p-N}r_{\epsilon}^{N}F^{o}(x_{\epsilon})^{\beta}\Psi_{\epsilon}(x)^{p-1}\notag\\
    =&c_{\epsilon}^{-N}F^{o}(x_{\epsilon})^{\beta}F^{o}(x_{\epsilon}+r_{\epsilon}F^{o}(x_{\epsilon})^{\frac{\beta}{N}}x)^{-\beta}\Psi_{\epsilon}(x)^{\frac{1}{N-1}}e^{\lambda_{N}(1-\frac{\beta}{N}-\epsilon)c_{\epsilon}^{\frac{N}{N-1}}(\Psi_{\epsilon}(x)^{\frac{N}{N-1}}-1)}
\end{align}
in $\Omega_{\epsilon}$. Since $0\leq\Psi_{\epsilon}\leq1$ and $F^{o}(x_{\epsilon})^{\beta}F^{o}(x_{\epsilon}+r_{\epsilon}F^{o}(x_{\epsilon})^{\frac{\beta}{N}}x)^{-\beta}=1+o_{\epsilon}(1)$, we have
\begin{align}
    &c_{\epsilon}^{-N}F^{o}(x_{\epsilon})^{\beta}F^{o}(x_{\epsilon}+r_{\epsilon}F^{o}(x_{\epsilon})^{\frac{\beta}{N}}x)^{-\beta}\Psi_{\epsilon}(x)^{\frac{1}{N-1}}e^{\lambda_{N}(1-\frac{\beta}{N}-\epsilon)c_{\epsilon}^{\frac{N}{N-1}}(\Psi_{\epsilon}(x)^{\frac{N}{N-1}}-1)}\notag\\
    \leq&c_{\epsilon}^{-N}+o_{\epsilon}(1)\rightarrow0\notag
\end{align}
and
\begin{align}
    &\left(\int_{\Omega_{\epsilon}}\left(\gamma\|u_{\epsilon}\|_{p}^{N-p}c_{\epsilon}^{p-N}r_{\epsilon}^{N}F^{o}(x_{\epsilon})^{\beta}\Psi_{\epsilon}(x)^{p-1}\right)^{\frac{p}{p-1}}\;\mathrm{d}x\right)^{\frac{p-1}{p}}\notag\\
    =&\gamma\|u_{\epsilon}\|_{p}^{N-1}c_{\epsilon}^{1-N}r_{\epsilon}^{\frac{N}{p}}F^{o}(x_{\epsilon})^{\frac{\beta}{p}}\rightarrow0\notag
\end{align}
as $\epsilon\rightarrow0$. Then applying elliptic regularity theory to \eqref{eq4.4}, we have $\Psi_{\epsilon}\rightarrow\Psi$ in $C^{0}(\R^{N})\cap C_{\mathrm{loc}}^{1}(\R^{N})$, where $\Psi$ is a weak solution to
\begin{align}\label{eq4.5}
    -Q_{N}(\Psi)=0
\end{align}
in $\R^{N}$. Since $0\leq\Psi\leq1$ and $\Psi(0)=1$, the Liouville theorem leads to $\Psi\equiv1$.

Also $\Phi_{\epsilon}$ is a weak solution to
\begin{align}\label{eq4.6}
    &-Q_{N}(\Phi_{\epsilon})(x)-\gamma\|u_{\epsilon}\|_{p}^{N-p}c_{\epsilon}^{p}r_{\epsilon}^{N}F^{o}(x_{\epsilon})^{\beta}\Psi_{\epsilon}(x)^{p-1}\notag\\
    =&F^{o}(x_{\epsilon})^{\beta}F^{o}(x_{\epsilon}+r_{\epsilon}F^{o}(x_{\epsilon})^{\frac{\beta}{N}}x)^{-\beta}\Psi_{\epsilon}(x)^{\frac{1}{N-1}}e^{\lambda_{N}(1-\frac{\beta}{N}-\epsilon)c_{\epsilon}^{\frac{N}{N-1}}(\Psi_{\epsilon}(x)^{\frac{N}{N-1}}-1)}
\end{align}
in $\Omega_{\epsilon}$. When $p>N$, for some $R>0$ and sufficiently small $\epsilon$,
\begin{align}
    \|u_{\epsilon}\|_{p}^{N-p}&\leq\left(\int_{W_{Rr_{\epsilon}F^{o}(x_{\epsilon})^{\beta/N}}^{x_{\epsilon}}}u_{\epsilon}(x)^{p}\;\mathrm{d}x\right)^{\frac{N-p}{p}}\notag\\
    &=c_{\epsilon}^{N-p}r_{\epsilon}^{\frac{N(N-p)}{p}}F^{o}(x_{\epsilon})^{\frac{\beta(N-p)}{p}}\left(\int_{W_{R}}\Psi_{\epsilon}(x)^{p}\;\mathrm{d}x\right)^{\frac{N-p}{p}}.\notag
\end{align}
Thus
\begin{align}
    \gamma\|u_{\epsilon}\|_{p}^{N-p}c_{\epsilon}^{p}r_{\epsilon}^{N}F^{o}(x_{\epsilon})^{\beta}\Psi_{\epsilon}(x)^{p-1}\leq\gamma c_{\epsilon}^{N}r_{\epsilon}^{\frac{N^{2}}{p}}F^{o}(x_{\epsilon})^{\frac{\beta N}{p}}\left(\int_{W_{R}}\Psi_{\epsilon}(x)^{p}\;\mathrm{d}x\right)^{\frac{N-p}{p}}\rightarrow0\notag
\end{align}
as $\epsilon\rightarrow0$. When $1<p\leq N$, obviously, $\gamma\|u_{\epsilon}\|_{p}^{N-p}c_{\epsilon}^{p}r_{\epsilon}^{N}F^{o}(x_{\epsilon})^{\beta}\Psi_{\epsilon}(x)^{p-1}\rightarrow0$ as $\epsilon\rightarrow0$. Meanwhile, by the mean value theorem,
\begin{align}
    c_{\epsilon}^{\frac{N}{N-1}}(\Psi_{\epsilon}(x)^{\frac{N}{N-1}}-1)&=u_{\epsilon}(x_{\epsilon}+r_{\epsilon}F^{o}(x_{\epsilon})^{\frac{\beta}{N}}x)^{\frac{N}{N-1}}-c_{\epsilon}^{\frac{N}{N-1}}\notag\\
    &=\frac{N}{N-1}\xi_{\epsilon}^{\frac{1}{N-1}}\left(u_{\epsilon}(x_{\epsilon}+r_{\epsilon}F^{o}(x_{\epsilon})^{\frac{\beta}{N}}x)-c_{\epsilon}\right)\notag\\
    &=\frac{N}{N-1}(\frac{\xi_{\epsilon}}{c_{\epsilon}})^{\frac{1}{N-1}}\Phi_{\epsilon}(x),\notag
\end{align}
where $\xi_{\epsilon}$ lies between $u_{\epsilon}(x_{\epsilon}+r_{\epsilon}F^{o}(x_{\epsilon})^{\frac{\beta}{N}}x)$ and $c_{\epsilon}$. Then applying elliptic regularity theory to \eqref{eq4.6}, we have $\Phi_{\epsilon}\rightarrow\Phi$ in $C^{0}(\R^{N})\cap C_{\mathrm{loc}}^{1}(\R^{N})$, where $\Phi$ is a weak solution to
\begin{align}\label{eq4.7}
    -Q_{N}(\Phi)=e^{\lambda_{N}(1-\frac{\beta}{N})\frac{N}{N-1}\Phi}
\end{align}
in $\R^{N}$ and $\Phi(0)=0=\sup_{\R^{N}}\Phi$.

On one hand, for any $R>0$,
\begin{align}
    &\int_{W_{R}}e^{\lambda_{N}(1-\frac{\beta}{N})\frac{N}{N-1}\Phi}\;\mathrm{d}x\notag\\
    =&\lim_{\epsilon\rightarrow0}\int_{W_{R}}e^{\lambda_{N}(1-\frac{\beta}{N}-\epsilon)(u_{\epsilon}(x_{\epsilon}+r_{\epsilon}F^{o}(x_{\epsilon})^{\frac{\beta}{N}}x)^{\frac{N}{N-1}}-c_{\epsilon}^{\frac{N}{N-1}})}\;\mathrm{d}x\notag\\
    =&\lim_{\epsilon\rightarrow0}\lambda_{\epsilon}^{-1}\int_{W_{Rr_{\epsilon}F^{o}(x_{\epsilon})^{\beta/N}}^{x_{\epsilon}}}c_{\epsilon}^{\frac{N}{N-1}}F^{o}(x_{\epsilon})^{-\beta}e^{\lambda_{N}(1-\frac{\beta}{N}-\epsilon)u_{\epsilon}(y)^{\frac{N}{N-1}}}\;\mathrm{d}y\notag\\
    =&\lim_{\epsilon\rightarrow0}\lambda_{\epsilon}^{-1}\int_{W_{Rr_{\epsilon}F^{o}(x_{\epsilon})^{\beta/N}}^{x_{\epsilon}}}u_{\epsilon}(y)^{\frac{N}{N-1}}F^{o}(y)^{-\beta}e^{\lambda_{N}(1-\frac{\beta}{N}-\epsilon)u_{\epsilon}(y)^{\frac{N}{N-1}}}\;\mathrm{d}y\notag\\
    \leq&1,\notag
\end{align}
which leads to
\begin{align}
    \int_{\R^{N}}e^{\lambda_{N}(1-\frac{\beta}{N})\frac{N}{N-1}\Phi}\;\mathrm{d}x\leq1.\notag
\end{align}
On the other hand, by the results in \cite{6}, we have
\begin{align}
    \int_{\R^{N}}e^{\lambda_{N}(1-\frac{\beta}{N})\frac{N}{N-1}\Phi}\;\mathrm{d}x=\left(1-\frac{\beta}{N}\right)^{1-N},\notag
\end{align}
which is a contradiction. Thus this case can not occur.

$\emph{\textbf{Case\;2.}}$ $\frac{F^{o}(x_{\epsilon})^{1-\frac{\beta}{N}}}{r_{\epsilon}}\leq C$ for some constant $C$.

Define
\begin{align}\label{eq4.8}
    \Psi_{\epsilon}(x)=c_{\epsilon}^{-1}u_{\epsilon}(x_{\epsilon}+r_{\epsilon}^{\frac{N}{N-\beta}}x),
\end{align}
and
\begin{align}\label{eq4.9}
    \Phi_{\epsilon}(x)=c_{\epsilon}^{\frac{N}{N-1}}(\Psi_{\epsilon}(x)-1)=c_{\epsilon}^{\frac{1}{N-1}}(u_{\epsilon}(x_{\epsilon}+r_{\epsilon}^{\frac{N}{N-\beta}}x)-c_{\epsilon}).
\end{align}
on $\Omega_{\epsilon}=\{x\in\R^{N}:x_{\epsilon}+r_{\epsilon}^{\frac{N}{N-\beta}}x\in\Omega\}$.

\begin{lemma}\label{le4.2}
    Let $\Psi_{\epsilon}$ and $\Phi_{\epsilon}$ be defined as in \eqref{eq4.8} and \eqref{eq4.9}, then

    (i) $\Psi_{\epsilon}\rightarrow1$ in $C^{0}(\R^{N})\cap C_{\mathrm{loc}}^{1}(\R^{N})$;

    (ii) $\Phi_{\epsilon}\rightarrow\Phi$ in $C^{0}(\R^{N})\cap C_{\mathrm{loc}}^{1}(\R^{N}\backslash\{0\})$, where
    $$
    \Phi(x)=-\frac{N-1}{\lambda_{N}(1-\frac{\beta}{N})}\log\left(1+\frac{\lambda_{N}F^{o}(x)^{\frac{N}{N-1}(1-\frac{\beta}{N})}}{N^{\frac{N}{N-1}}(1-\frac{\beta}{N})^{\frac{1}{N-1}}}\right).
    $$
    Moreover,
    $$
    \int_{\R^{N}}F^{o}(x)^{-\beta}e^{\lambda_{N}(1-\frac{\beta}{N})\frac{N}{N-1}\Phi}\;\mathrm{d}x=1.
    $$
\end{lemma}
\begin{proof}

    (i)A straightforward calculation shows  $\Psi_{\epsilon}$ is a weak solution to
    \begin{align}\label{eq4.10}
        &-Q_{N}(\Psi_{\epsilon})(x)-\gamma\|u_{\epsilon}\|_{p}^{N-p}c_{\epsilon}^{p-N}r_{\epsilon}^{\frac{N^{2}}{N-\beta}}\Psi_{\epsilon}(x)^{p-1}\notag\\
        =&c_{\epsilon}^{-N}F^{o}(x+r_{\epsilon}^{-\frac{N}{N-\beta}}x_{\epsilon})^{-\beta}\Psi_{\epsilon}(x)^{\frac{1}{N-1}}e^{\lambda_{N}(1-\frac{\beta}{N}-\epsilon)c_{\epsilon}^{\frac{N}{N-1}}(\Psi_{\epsilon}(x)^{\frac{N}{N-1}}-1)}
    \end{align}
    in $\Omega_{\epsilon}$. We can assume $r_{\epsilon}^{-\frac{N}{N-\beta}}x_{\epsilon}\rightarrow x^{*}$, then by $0\leq\Psi_{\epsilon}\leq1$ and Lemma \ref{le2.1},
    \begin{align}
        c_{\epsilon}^{-N}F^{o}(x+r_{\epsilon}^{-\frac{N}{N-\beta}}x_{\epsilon})^{-\beta}\Psi_{\epsilon}(x)^{\frac{1}{N-1}}e^{\lambda_{N}(1-\frac{\beta}{N}-\epsilon)c_{\epsilon}^{\frac{N}{N-1}}(\Psi_{\epsilon}(x)^{\frac{N}{N-1}}-1)}\rightarrow0\notag
    \end{align}
    and
    \begin{align}
        \left(\int_{\Omega_{\epsilon}}\left(\gamma\|u_{\epsilon}\|_{p}^{N-p}c_{\epsilon}^{p-N}r_{\epsilon}^{\frac{N^{2}}{N-\beta}}\Psi_{\epsilon}(x)^{p-1}\right)^{\frac{p}{p-1}}\;\mathrm{d}x\right)^{\frac{p-1}{p}}=\gamma\|u_{\epsilon}\|_{p}^{N-1}c_{\epsilon}^{1-N}r_{\epsilon}^{\frac{N^{2}}{p(N-\beta)}}\rightarrow0\notag
    \end{align}
    as $\epsilon\rightarrow0$. Then applying elliptic regularity theory to \eqref{eq4.10}, we have $\Psi_{\epsilon}\rightarrow\Psi$ in $C^{0}(\R^{N})\cap C_{\mathrm{loc}}^{1}(\R^{N})$, where $\Psi$ is a weak solution to
    \begin{align}\label{eq4.11}
        -Q_{N}(\Psi)=0
    \end{align}
    in $\R^{N}$. Since $0\leq\Psi\leq1$ and $\Psi(0)=1$, the Liouville theorem leads to $\Psi\equiv1$.

    (ii) Also $\Phi_{\epsilon}$ is a weak solution to
    \begin{align}\label{eq4.12}
        &-Q_{N}(\Phi_{\epsilon})(x)-\gamma\|u_{\epsilon}\|_{p}^{N-p}c_{\epsilon}^{p}r_{\epsilon}^{\frac{N^{2}}{N-\beta}}\Psi_{\epsilon}(x)^{p-1}\notag\\
        =&F^{o}(x+r_{\epsilon}^{-\frac{N}{N-\beta}}x_{\epsilon})^{-\beta}\Psi_{\epsilon}(x)^{\frac{1}{N-1}}e^{\lambda_{N}(1-\frac{\beta}{N}-\epsilon)c_{\epsilon}^{\frac{N}{N-1}}(\Psi_{\epsilon}(x)^{\frac{N}{N-1}}-1)}
    \end{align}
    in $\Omega_{\epsilon}$. When $p>N$, for some $R>0$ and sufficiently small $\epsilon$,
    \begin{align}
        \|u_{\epsilon}\|_{p}^{N-p}&\leq\left(\int_{W_{Rr_{\epsilon}^{N/(N-\beta)}}^{x_{\epsilon}}}u_{\epsilon}(x)^{p}\;\mathrm{d}x\right)^{\frac{N-p}{p}}\notag\\
        &=c_{\epsilon}^{N-p}r_{\epsilon}^{\frac{N^{2}(N-p)}{(N-\beta)p}}\left(\int_{W_{R}}\Psi_{\epsilon}(x)^{p}\;\mathrm{d}x\right)^{\frac{N-p}{p}}.\notag
    \end{align}
    Thus
    \begin{align}
        \gamma\|u_{\epsilon}\|_{p}^{N-p}c_{\epsilon}^{p}r_{\epsilon}^{\frac{N^{2}}{N-\beta}}\Psi_{\epsilon}(x)^{p-1}\leq\gamma c_{\epsilon}^{N}r_{\epsilon}^{\frac{N^{3}}{(N-\beta)p}}\left(\int_{W_{R}}\Psi_{\epsilon}(x)^{p}\;\mathrm{d}x\right)^{\frac{N-p}{p}}\rightarrow0\notag
    \end{align}
    as $\epsilon\rightarrow0$. When $1<p\leq N$, $\gamma\|u_{\epsilon}\|_{p}^{N-p}c_{\epsilon}^{p}r_{\epsilon}^{\frac{N^{2}}{N-\beta}}\Psi_{\epsilon}(x)^{p-1}\rightarrow0$ as $\epsilon\rightarrow0$ is obvious. Meanwhile, by the mean value theorem,
    \begin{align}\label{eq4.13}
        c_{\epsilon}^{\frac{N}{N-1}}(\Psi_{\epsilon}(x)^{\frac{N}{N-1}}-1)&=u_{\epsilon}(x_{\epsilon}+r_{\epsilon}^{\frac{N}{N-\beta}}x)^{\frac{N}{N-1}}-c_{\epsilon}^{\frac{N}{N-1}}\notag\\
        &=\frac{N}{N-1}\xi_{\epsilon}^{\frac{1}{N-1}}\left(u_{\epsilon}(x_{\epsilon}+r_{\epsilon}^{\frac{N}{N-\beta}}x)-c_{\epsilon}\right)\notag\\
        &=\frac{N}{N-1}(\frac{\xi_{\epsilon}}{c_{\epsilon}})^{\frac{1}{N-1}}\Phi_{\epsilon}(x),
    \end{align}
    where $\xi_{\epsilon}$ lies between $u_{\epsilon}(x_{\epsilon}+r_{\epsilon}^{\frac{N}{N-\beta}}x)$ and $c_{\epsilon}$. Applying elliptic regularity theory to \eqref{eq4.12}, we have $\Phi_{\epsilon}\rightarrow\Phi$ in $C^{0}(\R^{N})\cap C_{\mathrm{loc}}^{1}(\R^{N}\backslash\{-x^{*}\})$, where $\Phi$ is a weak solution to
\begin{align}\label{eq4.14}
    -Q_{N}(\Phi)=F^{o}(x+x^{*})^{-\beta}e^{\lambda_{N}(1-\frac{\beta}{N})\frac{N}{N-1}\Phi}
\end{align}
in $\R^{N}\backslash \{-x^{*}\}$.

For any $R>F^{o}(x^{*})+1$,
\begin{align}
    &\int_{W_{R}^{-x^{*}}}F^{o}(x+x^{*})^{-\beta}e^{\lambda_{N}(1-\frac{\beta}{N})\frac{N}{N-1}\Phi}\;\mathrm{d}x\notag\\
    =&\lim_{\epsilon\rightarrow0}\int_{W_{R}^{-x^{*}}}F^{o}(x+r_{\epsilon}^{-\frac{N}{N-\beta}}x_{\epsilon})^{-\beta}e^{\lambda_{N}(1-\frac{\beta}{N}-\epsilon)(u_{\epsilon}(x_{\epsilon}+r_{\epsilon}^{\frac{N}{N-\beta}}x)^{\frac{N}{N-1}}-c_{\epsilon}^{\frac{N}{N-1}})}\;\mathrm{d}x\notag\\
    =&\lim_{\epsilon\rightarrow0}\lambda_{\epsilon}^{-1}\int_{W_{2Rr_{\epsilon}^{N/(N-\beta)}}}c_{\epsilon}^{\frac{N}{N-1}}F^{o}(y)^{-\beta}e^{\lambda_{N}(1-\frac{\beta}{N}-\epsilon)u_{\epsilon}(y)^{\frac{N}{N-1}}}\;\mathrm{d}y\notag\\
    =&\lim_{\epsilon\rightarrow0}\lambda_{\epsilon}^{-1}\int_{W_{2Rr_{\epsilon}^{N/(N-\beta)}}}u_{\epsilon}(y)^{\frac{N}{N-1}}F^{o}(y)^{-\beta}e^{\lambda_{N}(1-\frac{\beta}{N}-\epsilon)u_{\epsilon}(y)^{\frac{N}{N-1}}}\;\mathrm{d}y\notag\\
    \leq&1,\notag
\end{align}
which leads to
\begin{align}
    \int_{\R^{N}}F^{o}(x+x^{*})^{-\beta}e^{\lambda_{N}(1-\frac{\beta}{N})\frac{N}{N-1}\Phi}\;\mathrm{d}x\leq1.\notag
\end{align}
Thus by Theorem 4.1 in \cite{Lu-Shen-Xue-Zhu}, we obtain $x^{*}=0$,
$$
\Phi(x)=-\frac{N-1}{\lambda_{N}(1-\frac{\beta}{N})}\log\left(1+\frac{\lambda_{N}F^{o}(x)^{\frac{N}{N-1}(1-\frac{\beta}{N})}}{N^{\frac{N}{N-1}}(1-\frac{\beta}{N})^{\frac{1}{N-1}}}\right)
$$
and
$$
\int_{\R^{N}}F^{o}(x)^{-\beta}e^{\lambda_{N}(1-\frac{\beta}{N})\frac{N}{N-1}\Phi}\;\mathrm{d}x=1.
$$
\end{proof}

Define $u_{\epsilon,a}=\min\{u_{\epsilon},ac_{\epsilon}\}$. Then we have the following:
\begin{lemma}\label{le4.3}
    For any $0<a<1$, there holds
    $$
    \lim_{\epsilon\rightarrow0}\int_{\Omega}F(\nabla u_{\epsilon,a})^{N}\;\mathrm{d}x=a.
    $$
\end{lemma}
\begin{proof}
    By the equation \eqref{eq3.2}, Lemma \ref{le2.1} and Lemma \ref{le4.2}, we have for any $R>0$,
    \begin{align}
        \int_{\Omega}F(\nabla u_{\epsilon,a})^{N}\;\mathrm{d}x=&\int_{\Omega}F(\nabla u_{\epsilon})^{N-1}\nabla F(\nabla u_{\epsilon})\cdot\nabla u_{\epsilon,a}\;\mathrm{d}x\notag\\
        =&-\int_{\Omega}\mathrm{div}\left(F(\nabla u_{\epsilon})^{N-1}\nabla F(\nabla u_{\epsilon})\right)u_{\epsilon,a}\;\mathrm{d}x\notag\\
        =&\int_{\Omega}\lambda_{\epsilon}^{-1}F^{o}(x)^{-\beta}u_{\epsilon}^{\frac{1}{N-1}}e^{\lambda_{N}(1-\frac{\beta}{N}-\epsilon)u_{\epsilon}^{\frac{N}{N-1}}}u_{\epsilon,a}\;\mathrm{d}x\notag\\
        &+\int_{\Omega}\gamma\|u_{\epsilon}\|_{p}^{N-p}u_{\epsilon}^{p-1}u_{\epsilon,a}\;\mathrm{d}x\notag\\
        =&\int_{\{u_{\epsilon}\leq ac_{\epsilon}\}}\lambda_{\epsilon}^{-1}F^{o}(x)^{-\beta}u_{\epsilon}^{\frac{1}{N-1}}e^{\lambda_{N}(1-\frac{\beta}{N}-\epsilon)u_{\epsilon}^{\frac{N}{N-1}}}u_{\epsilon,a}\;\mathrm{d}x\notag\\
        &+\int_{\{u_{\epsilon}>ac_{\epsilon}\}}\lambda_{\epsilon}^{-1}F^{o}(x)^{-\beta}u_{\epsilon}^{\frac{1}{N-1}}e^{\lambda_{N}(1-\frac{\beta}{N}-\epsilon)u_{\epsilon}^{\frac{N}{N-1}}}u_{\epsilon,a}\;\mathrm{d}x\notag\\
        &+\int_{\Omega}\gamma\|u_{\epsilon}\|_{p}^{N-p}u_{\epsilon}^{p-1}u_{\epsilon,a}\;\mathrm{d}x\notag\\
        \geq&\int_{W_{Rr_{\epsilon}^{N/(N-\beta)}}^{x_{\epsilon}}}ac_{\epsilon}\lambda_{\epsilon}^{-1}F^{o}(x)^{-\beta}u_{\epsilon}^{\frac{1}{N-1}}e^{\lambda_{N}(1-\frac{\beta}{N}-\epsilon)u_{\epsilon}^{\frac{N}{N-1}}}\;\mathrm{d}x+o_{\epsilon}(1)\notag
    \end{align}
    for $\epsilon>0$ small enough. Making the change in variable $x=x_{\epsilon}+r_{\epsilon}^{\frac{N}{N-\beta}}y$, we get
    \begin{align}
        &\int_{W_{Rr_{\epsilon}^{N/(N-\beta)}}^{x_{\epsilon}}}ac_{\epsilon}\lambda_{\epsilon}^{-1}F^{o}(x)^{-\beta}u_{\epsilon}^{\frac{1}{N-1}}e^{\lambda_{N}(1-\frac{\beta}{N}-\epsilon)u_{\epsilon}^{\frac{N}{N-1}}}\;\mathrm{d}x\notag\\
        =&\int_{W_{R}}aF^{o}(y+r_{\epsilon}^{-\frac{N}{N-\beta}}x_{\epsilon})^{-\beta}\Psi_{\epsilon}(y)^{\frac{1}{N-1}}e^{\lambda_{N}(1-\frac{\beta}{N}-\epsilon)c_{\epsilon}^{\frac{N}{N-1}}(\Psi_{\epsilon}(y)^{\frac{N}{N-1}}-1)}\;\mathrm{d}y.\notag
    \end{align}
    Thus by Lemma \ref{le4.2} and \eqref{eq4.13},
    \begin{align}
        \liminf_{\epsilon\rightarrow0}\int_{\Omega}F(\nabla u_{\epsilon,a})^{N}\;\mathrm{d}x\geq\int_{W_{R}}aF^{o}(x)^{-\beta}e^{\lambda_{N}(1-\frac{\beta}{N})\frac{N}{N-1}\Phi}\;\mathrm{d}x.\notag
    \end{align}
    Letting $R\rightarrow+\infty$, we obtain
    $$
    \liminf_{\epsilon\rightarrow0}\int_{\Omega}F(\nabla u_{\epsilon,a})^{N}\;\mathrm{d}x\geq a.
    $$
    Similarly we have
    $$
    \liminf_{\epsilon\rightarrow0}\int_{\Omega}F(\nabla(u_{\epsilon}-u_{\epsilon,a})))^{N}\;\mathrm{d}x\geq 1-a.
    $$
    Combining with
    \begin{align}
        \int_{\Omega}F(\nabla u_{\epsilon,a})^{N}\;\mathrm{d}x&=\int_{\Omega}F(\nabla u_{\epsilon})^{N}\;\mathrm{d}x-\int_{\Omega}F(\nabla(u_{\epsilon}-u_{\epsilon,a})))^{N}\;\mathrm{d}x\notag\\
        &=1+\gamma\|u_{\epsilon}\|_{p}^{N}-\int_{\Omega}F(\nabla(u_{\epsilon}-u_{\epsilon,a})))^{N}\;\mathrm{d}x,\notag
    \end{align}
    we obtain
    $$
    \limsup_{\epsilon\rightarrow0}\int_{\Omega}F(\nabla u_{\epsilon,a})^{N}\;\mathrm{d}x\leq a.
    $$
    Therefore,
    $$
    \lim_{\epsilon\rightarrow0}\int_{\Omega}F(\nabla u_{\epsilon,a})^{N}\;\mathrm{d}x=a.
    $$
\end{proof}

As a consequence of Lemma \ref{le4.3}, we have the following:
\begin{lemma}\label{le4.4}
    $$
    \Lambda=\lim_{\epsilon\rightarrow0}\int_{\Omega}\frac{e^{\lambda_{N}(1-\frac{\beta}{N}-\epsilon)u_{\epsilon}^{\frac{N}{N-1}}}}{F^{o}(x)^{\beta}}\;\mathrm{d}x\leq\int_{\Omega}\frac{1}{F^{o}(x)^{\beta}}\;\mathrm{d}x+\limsup_{\epsilon\rightarrow0}\frac{\lambda_{\epsilon}}{c_{\epsilon}^{N/(N-1)}}.
    $$
\end{lemma}
\begin{proof}
    For any $0<a<1$, there holds
    \begin{align}\label{eq4.15}
        &\int_{\Omega}\frac{e^{\lambda_{N}(1-\frac{\beta}{N}-\epsilon)u_{\epsilon}^{\frac{N}{N-1}}}}{F^{o}(x)^{\beta}}\;\mathrm{d}x\notag\\
        =&\int_{\{u_{\epsilon}\leq ac_{\epsilon}\}}\frac{e^{\lambda_{N}(1-\frac{\beta}{N}-\epsilon)u_{\epsilon}^{\frac{N}{N-1}}}}{F^{o}(x)^{\beta}}\;\mathrm{d}x+\int_{\{u_{\epsilon}>ac_{\epsilon}\}}\frac{e^{\lambda_{N}(1-\frac{\beta}{N}-\epsilon)u_{\epsilon}^{\frac{N}{N-1}}}}{F^{o}(x)^{\beta}}\;\mathrm{d}x\notag\\
        <&\int_{\{u_{\epsilon}\leq ac_{\epsilon}\}}\frac{e^{\lambda_{N}(1-\frac{\beta}{N}-\epsilon)u_{\epsilon,a}^{\frac{N}{N-1}}}}{F^{o}(x)^{\beta}}\;\mathrm{d}x+(ac_{\epsilon})^{-\frac{N}{N-1}}\int_{\{u_{\epsilon}>ac_{\epsilon}\}}u_{\epsilon}^{\frac{N}{N-1}}\frac{e^{\lambda_{N}(1-\frac{\beta}{N}-\epsilon)u_{\epsilon}^{\frac{N}{N-1}}}}{F^{o}(x)^{\beta}}\;\mathrm{d}x\notag\\
        \leq&\int_{\Omega}\frac{e^{\lambda_{N}(1-\frac{\beta}{N}-\epsilon)u_{\epsilon,a}^{\frac{N}{N-1}}}}{F^{o}(x)^{\beta}}\;\mathrm{d}x+\frac{\lambda_{\epsilon}}{(ac_{\epsilon})^{N/(N-1)}}.
    \end{align}
    By Theorem 2.1 in \cite{35} and Lemma \ref{le4.3}, $F^{o}(x)^{-\beta}e^{\lambda_{N}(1-\frac{\beta}{N}-\epsilon)u_{\epsilon,a}^{\frac{N}{N-1}}}$ is bounded in $L^{i}(\Omega)$ for some $i>1$. Notice that $u_{\epsilon,a}$ converges to $0$ almost everywhere in $\Omega$, hence $F^{o}(x)^{-\beta}e^{\lambda_{N}(1-\frac{\beta}{N}-\epsilon)u_{\epsilon,a}^{\frac{N}{N-1}}}$ converges to $F^{o}(x)^{-\beta}$ in $L^{1}(\Omega)$. Letting $\epsilon\rightarrow0$ and $a\rightarrow1$ in \eqref{eq4.15}, we conclude the lemma.
\end{proof}

It follows from Lemma \ref{le4.4} that for any $\theta<\frac{N}{N-1}$,
\begin{align}\label{eq4.16}
    \limsup_{\epsilon\rightarrow0}\frac{\lambda_{\epsilon}}{c_{\epsilon}^{\theta}}=+\infty.
\end{align}
In fact, $\limsup\limits_{\epsilon\rightarrow0}\frac{\lambda_{\epsilon}}{c_{\epsilon}^{N/(N-1)}}>0$ and notice that
$$\limsup_{\epsilon\rightarrow0}\frac{\lambda_{\epsilon}}{c_{\epsilon}^{\theta}}\geq\limsup_{\epsilon\rightarrow0}\frac{\lambda_{\epsilon}}{c_{\epsilon}^{N/(N-1)}}\cdot\liminf_{\epsilon\rightarrow0}c_{\epsilon}^{\frac{N}{N-1}-\theta}.$$
\begin{lemma}\label{le4.5}
    For any $1<i<N$, $c_{\epsilon}^{\frac{1}{N-1}}u_{\epsilon}\rightharpoonup G$ weakly in $W_{0}^{1,i}(\Omega)$, where $G$ is a weak solution to
    \begin{align}\label{eq4.17}
        -Q_{N}(G)=\delta_{0}+\gamma\|G\|_{p}^{N-p}G^{p-1}
    \end{align}
    in $\Omega$. Moreover, $c_{\epsilon}^{\frac{1}{N-1}}u_{\epsilon}\rightarrow G$ in $C^{0}(\Omega)\cap C_{\mathrm{loc}}^{1}(\Omega\backslash\{0\})$ and $G$ takes the form
    \begin{align}
        G(x)=-\frac{N}{\lambda_{N}}\log F^{o}(x)+A_{0}+\xi(x),\notag
    \end{align}
    where $A_{0}$ is a constant and $\xi\in C^{1}(\overline{\Omega})$ with $\xi(x)=O(F^{o}(x))$ as $x\rightarrow0$.
\end{lemma}
\begin{proof}
     We claim  for any $\varphi\in C(\overline{\Omega})$, there holds
     \begin{align}\label{eq4.18}
         \lim_{\epsilon\rightarrow0}\int_{\Omega}c_{\epsilon}\lambda_{\epsilon}^{-1}F^{o}(x)^{-\beta}u_{\epsilon}^{\frac{1}{N-1}}e^{\lambda_{N}(1-\frac{\beta}{N}-\epsilon)u_{\epsilon}^{\frac{N}{N-1}}}\varphi\;\mathrm{d}x=\varphi(0).
     \end{align}
     In fact, for any $R>0$ and $0<a<1$,
     \begin{align}
         &\int_{\Omega}c_{\epsilon}\lambda_{\epsilon}^{-1}F^{o}(x)^{-\beta}u_{\epsilon}^{\frac{1}{N-1}}e^{\lambda_{N}(1-\frac{\beta}{N}-\epsilon)u_{\epsilon}^{\frac{N}{N-1}}}\varphi\;\mathrm{d}x\notag\\
         =&\int_{\{u_{\epsilon}\leq ac_{\epsilon}\}}c_{\epsilon}\lambda_{\epsilon}^{-1}F^{o}(x)^{-\beta}u_{\epsilon}^{\frac{1}{N-1}}e^{\lambda_{N}(1-\frac{\beta}{N}-\epsilon)u_{\epsilon}^{\frac{N}{N-1}}}\varphi\;\mathrm{d}x\notag\\
         &+\int_{W_{Rr_{\epsilon}^{N/(N-\beta)}}^{x_{\epsilon}}}c_{\epsilon}\lambda_{\epsilon}^{-1}F^{o}(x)^{-\beta}u_{\epsilon}^{\frac{1}{N-1}}e^{\lambda_{N}(1-\frac{\beta}{N}-\epsilon)u_{\epsilon}^{\frac{N}{N-1}}}\varphi\;\mathrm{d}x\notag\\
         &+\int_{\{u_{\epsilon}>ac_{\epsilon}\}\backslash W_{Rr_{\epsilon}^{N/(N-\beta)}}^{x_{\epsilon}}}c_{\epsilon}\lambda_{\epsilon}^{-1}F^{o}(x)^{-\beta}u_{\epsilon}^{\frac{1}{N-1}}e^{\lambda_{N}(1-\frac{\beta}{N}-\epsilon)u_{\epsilon}^{\frac{N}{N-1}}}\varphi\;\mathrm{d}x\notag
     \end{align}
     for $\epsilon>0$ small enough.

     We estimate the three integrals on the right hand respectively. By Theorem 2.1 in \cite{35}, Lemma \ref{le4.3} and \eqref{eq4.16},
     \begin{align}
         \int_{\{u_{\epsilon}\leq ac_{\epsilon}\}}c_{\epsilon}\lambda_{\epsilon}^{-1}F^{o}(x)^{-\beta}u_{\epsilon}^{\frac{1}{N-1}}e^{\lambda_{N}(1-\frac{\beta}{N}-\epsilon)u_{\epsilon}^{\frac{N}{N-1}}}\varphi\;\mathrm{d}x=o_{\epsilon}(1).\notag
     \end{align}
     By Lemma \ref{le4.2} and \eqref{eq4.13}, we have
     \begin{align}
         &\int_{W_{Rr_{\epsilon}^{N/(N-\beta)}}^{x_{\epsilon}}}c_{\epsilon}\lambda_{\epsilon}^{-1}F^{o}(x)^{-\beta}u_{\epsilon}^{\frac{1}{N-1}}e^{\lambda_{N}(1-\frac{\beta}{N}-\epsilon)u_{\epsilon}^{\frac{N}{N-1}}}\varphi\;\mathrm{d}x\notag\\
         =&\varphi(0)(1+o_{\epsilon}(1))\int_{W_{Rr_{\epsilon}^{N/(N-\beta)}}^{x_{\epsilon}}}c_{\epsilon}\lambda_{\epsilon}^{-1}F^{o}(x)^{-\beta}u_{\epsilon}^{\frac{1}{N-1}}e^{\lambda_{N}(1-\frac{\beta}{N}-\epsilon)u_{\epsilon}^{\frac{N}{N-1}}}\;\mathrm{d}x\notag\\
         =&\varphi(0)(1+o_{\epsilon}(1))\int_{W_{R}}F^{o}(x+r_{\epsilon}^{-\frac{N}{N-\beta}}x_{\epsilon})^{-\beta}\Psi_{\epsilon}(x)^{\frac{1}{N-1}}e^{\lambda_{N}(1-\frac{\beta}{N}-\epsilon)c_{\epsilon}^{\frac{N}{N-1}}(\Psi_{\epsilon}(x)^{\frac{N}{N-1}}-1)}\;\mathrm{d}x\notag\\
         \rightarrow&\varphi(0)\int_{W_{R}}F^{o}(x)^{-\beta}e^{\lambda_{N}(1-\frac{\beta}{N})\frac{N}{N-1}\Phi}\;\mathrm{d}x\notag
     \end{align}
     and
     \begin{align}
         &\int_{\{u_{\epsilon}>ac_{\epsilon}\}\backslash W_{Rr_{\epsilon}^{N/(N-\beta)}}^{x_{\epsilon}}}c_{\epsilon}\lambda_{\epsilon}^{-1}F^{o}(x)^{-\beta}u_{\epsilon}^{\frac{1}{N-1}}e^{\lambda_{N}(1-\frac{\beta}{N}-\epsilon)u_{\epsilon}^{\frac{N}{N-1}}}\varphi\;\mathrm{d}x\notag\\
         \leq&\|\varphi\|_{\infty}\left(\frac{1}{a}-\int_{W_{R}}F^{o}(x+r_{\epsilon}^{-\frac{N}{N-\beta}}x_{\epsilon})^{-\beta}\Psi_{\epsilon}(x)^{\frac{1}{N-1}}e^{\lambda_{N}(1-\frac{\beta}{N}-\epsilon)c_{\epsilon}^{\frac{N}{N-1}}(\Psi_{\epsilon}(x)^{\frac{N}{N-1}}-1)}\;\mathrm{d}x\right)\notag\\
         \rightarrow&\|\varphi\|_{\infty}\left(\frac{1}{a}-\int_{W_{R}}F^{o}(x)^{-\beta}e^{\lambda_{N}(1-\frac{\beta}{N})\frac{N}{N-1}\Phi}\;\mathrm{d}x\right)\notag
     \end{align}
     as $\epsilon\rightarrow0$. Letting $R\rightarrow+\infty$ and $a\rightarrow1$, we conclude \eqref{eq4.18}.

     In view of the equation \eqref{eq3.2}, $c_{\epsilon}^{\frac{1}{N-1}}u_{\epsilon}$ is a weak solution to
     \begin{align}\label{eq4.19}
         -Q_{N}(c_{\epsilon}^{\frac{1}{N-1}}u_{\epsilon})-c_{\epsilon}\gamma\|u_{\epsilon}\|_{p}^{N-p}u_{\epsilon}^{p-1}=c_{\epsilon}\lambda_{\epsilon}^{-1}F^{o}(x)^{-\beta}u_{\epsilon}^{\frac{1}{N-1}}e^{\lambda_{N}(1-\frac{\beta}{N}-\epsilon)u_{\epsilon}^{\frac{N}{N-1}}}
     \end{align}
     in $\Omega$. By \eqref{eq4.18}, $c_{\epsilon}\lambda_{\epsilon}^{-1}F^{o}(x)^{-\beta}u_{\epsilon}^{\frac{1}{N-1}}e^{\lambda_{N}(1-\frac{\beta}{N}-\epsilon)u_{\epsilon}^{\frac{N}{N-1}}}$ is bounded in $L^{1}(\Omega)$. Then by Lemma 2.2 in \cite{17}, $c_{\epsilon}^{\frac{1}{N-1}}u_{\epsilon}$ is bounded in $W_{0}^{1,i}(\Omega)$ for any $1<i<N$. Hence there exists some $G\in W_{0}^{1,i}(\Omega)$ such that $c_{\epsilon}^{\frac{1}{N-1}}u_{\epsilon}\rightharpoonup G$ weakly in $W_{0}^{1,i}(\Omega)$ for any $1<i<N$. Since \eqref{eq4.18} implies  $c_{\epsilon}\lambda_{\epsilon}^{-1}F^{o}(x)^{-\beta}u_{\epsilon}^{\frac{1}{N-1}}e^{\lambda_{N}(1-\frac{\beta}{N}-\epsilon)u_{\epsilon}^{\frac{N}{N-1}}}\rightharpoonup \delta_{0}$ weakly in the sense of measure as $\epsilon\rightarrow0$, in view of \eqref{eq4.19}, $G$ is a weak solution to \eqref{eq4.17}. Applying elliptic regularity theory, we get $c_{\epsilon}^{\frac{1}{N-1}}u_{\epsilon}\rightarrow G$ in $C^{0}(\Omega)\cap C_{\mathrm{loc}}^{1}(\Omega\backslash\{0\})$. Using the argument in the proof of Proposition 5.1 in \cite{30}, we obtain the form of $G$.
\end{proof}

\subsection{Capacity estimate}

In this subsection, we use the capacity estimate, which was first used by Li \cite{13} in this topic, to derive an upper bound of the integrals $\int_{\Omega}F^{o}(x)^{-\beta}e^{\lambda_{N}(1-\frac{\beta}{N}-\epsilon)u_{\epsilon}^{\frac{N}{N-1}}}\;\mathrm{d}x$.
\begin{lemma}\label{le4.6}
    $$
    \Lambda=\lim_{\epsilon\rightarrow0}\int_{\Omega}\frac{e^{\lambda_{N}(1-\frac{\beta}{N}-\epsilon)u_{\epsilon}^{\frac{N}{N-1}}}}{F^{o}(x)^{\beta}}\;\mathrm{d}x\leq\int_{\Omega}\frac{1}{F^{o}(x)^{\beta}}\;\mathrm{d}x+\frac{N\kappa_{N}}{N-\beta}e^{\lambda_{N}(1-\frac{\beta}{N})A_{0}+\sum_{k=1}^{N-1}\frac{1}{k}}.
    $$
\end{lemma}
\begin{proof}
    For any $R>0$, take $\delta>0$ such that $W_{2\delta}\subset\Omega$ and $\epsilon>0$ such that $Rr_{\epsilon}^{\frac{N}{N-\beta}}<\delta$. For $a,b\in\R$, we define a function space
    $$
    W_{\epsilon}(a,b)=\{u\in W^{1,N}(W_{\delta}^{x_{\epsilon}}\backslash\overline{W_{Rr_{\epsilon}^{N/(N-\beta)}}^{x_{\epsilon}}}):u|_{\partial W_{\delta}^{x_{\epsilon}}}=a,u|_{\partial W_{Rr_{\epsilon}^{N/(N-\beta)}}^{x_{\epsilon}}}=b\}.
    $$
    Let
    $$
    s_{\epsilon}=\sup_{\partial W_{\delta}^{x_{\epsilon}}}u_{\epsilon}\;\mathrm{and}\;i_{\epsilon}=\inf_{\partial W_{Rr_{\epsilon}^{N/(N-\beta)}}^{x_{\epsilon}}}u_{\epsilon}.
    $$
    By Lemma \ref{le4.2} and Lemma \ref{le4.5},
    $$
    s_{\epsilon}=c_{\epsilon}^{-\frac{1}{N-1}}\left(-\frac{N}{\lambda_{N}}\log\delta+A_{0}+o_{\delta}(1)+o_{\epsilon}(\delta)\right),
    $$
    $$
    i_{\epsilon}=c_{\epsilon}+c_{\epsilon}^{-\frac{1}{N-1}}\left(-\frac{N-1}{\lambda_{N}(1-\frac{\beta}{N})}\log(1+\frac{\lambda_{N}R^{\frac{N}{N-1}(1-\frac{\beta}{N})}}{N^{\frac{N}{N-1}}(1-\frac{\beta}{N})^{\frac{1}{N-1}}})+o_{\epsilon}(R^{-1})\right).
    $$
    Therefore $i_{\epsilon}>s_{\epsilon}$ for $\epsilon>0$ small enough.

    It is not difficult to see
    $$
    \inf_{u\in W_{\epsilon}(s_{\epsilon},i_{\epsilon})}\int_{W_{\delta}^{x_{\epsilon}}\backslash\overline{W_{Rr_{\epsilon}^{N/(N-\beta)}}^{x_{\epsilon}}}}F(\nabla u)^{N}\;\mathrm{d}x
    $$
    is attained by $h(x)$ satisfying
    \begin{align}
    \left\{\begin{array}{l}
    -Q_{N}(h)=0\;\;\;\;\;\mathrm{in}\;W_{\delta}^{x_{\epsilon}}\backslash\overline{W_{Rr_{\epsilon}^{N/(N-\beta)}}^{x_{\epsilon}}},\\
    h_{\partial W_{\delta}^{x_{\epsilon}}}=s_{\epsilon},\\
    h_{\partial W_{Rr_{\epsilon}^{N/(N-\beta)}}^{x_{\epsilon}}}=i_{\epsilon}.
    \end{array}\right.\notag
    \end{align}
    Then by the uniqueness of solution,
    \begin{align}
        h(x)=\frac{s_{\epsilon}(\log F^{o}(x-x_{\epsilon}))-\log(Rr_{\epsilon}^{N/(N-\beta)}))+i_{\epsilon}(\log\delta-\log F^{o}(x-x_{\epsilon}))}{\log\delta-\log(Rr_{\epsilon}^{N/(N-\beta)})}\notag
    \end{align}
    and
    \begin{align}\label{eq4.20}
        \int_{W_{\delta}^{x_{\epsilon}}\backslash\overline{W_{Rr_{\epsilon}^{N/(N-\beta)}}^{x_{\epsilon}}}}F(\nabla h)^{N}\;\mathrm{d}x=\frac{N\kappa_{N}(i_{\epsilon}-s_{\epsilon})^{N}}{(\log\delta-\log(Rr_{\epsilon}^{N/(N-\beta)}))^{N-1}}.
    \end{align}
    Notice that by the inequality $(1+t)^{m}\geq1+mt$ for $t\geq-1$ and $m\geq1$,
    \begin{align}\label{eq4.21}
        (i_{\epsilon}-s_{\epsilon})^{\frac{N}{N-1}}=&c_{\epsilon}^{\frac{N}{N-1}}\left(1+c_{\epsilon}^{-\frac{N}{N-1}}\left(-\frac{N-1}{\lambda_{N}(1-\frac{\beta}{N})}\log(1+\frac{\lambda_{N}R^{\frac{N}{N-1}(1-\frac{\beta}{N})}}{N^{\frac{N}{N-1}}(1-\frac{\beta}{N})^{\frac{1}{N-1}}})\right.\right.\notag\\
        &\left.\left.+\frac{N}{\lambda_{N}}\log\delta-A_{0}+o_{\epsilon}(R^{-1})+o_{\delta}(1)+o_{\epsilon}(\delta)\right)\right)^{\frac{N}{N-1}}\notag\\
        \geq&c_{\epsilon}^{\frac{N}{N-1}}+\frac{N}{N-1}\left(-\frac{N-1}{\lambda_{N}(1-\frac{\beta}{N})}\log(1+\frac{\lambda_{N}R^{\frac{N}{N-1}(1-\frac{\beta}{N})}}{N^{\frac{N}{N-1}}(1-\frac{\beta}{N})^{\frac{1}{N-1}}})\right.\notag\\
        &\left.+\frac{N}{\lambda_{N}}\log\delta-A_{0}+o_{\epsilon}(R^{-1})+o_{\delta}(1)+o_{\epsilon}(\delta)\right)
    \end{align}
    and by \eqref{eq4.2},
    \begin{align}\label{eq4.22}
        \log\delta-\log(Rr_{\epsilon}^{N/(N-\beta)})=\log\frac{\delta}{R}-\frac{1}{N-\beta}\log(\lambda_{\epsilon}c_{\epsilon}^{-\frac{N}{N-1}})+\frac{\lambda_{N}(1-\frac{\beta}{N}-\epsilon)c_{\epsilon}^{\frac{N}{N-1}}}{N-\beta}.
    \end{align}

    Now define $\overline{u}_{\epsilon}=\max\{s_{\epsilon},\min\{i_{\epsilon},u_{\epsilon}\}\}$. Then $\overline{u}_{\epsilon}\in W_{\epsilon}(s_{\epsilon},i_{\epsilon})$ and $F(\nabla\overline{u}_{\epsilon})\leq F(\nabla u_{\epsilon})$ almost everywhere in $W_{\delta}^{x_{\epsilon}}\backslash\overline{W_{Rr_{\epsilon}^{N/(N-\beta)}}^{x_{\epsilon}}}$ for $\epsilon>0$ small enough. Therefore
    \begin{align}\label{eq4.23}
        &\int_{W_{\delta}^{x_{\epsilon}}\backslash\overline{W_{Rr_{\epsilon}^{N/(N-\beta)}}^{x_{\epsilon}}}}F(\nabla h)^{N}\;\mathrm{d}x\notag\\
        \leq&\int_{W_{\delta}^{x_{\epsilon}}\backslash\overline{W_{Rr_{\epsilon}^{N/(N-\beta)}}^{x_{\epsilon}}}}F(\nabla\overline{u}_{\epsilon})^{N}\;\mathrm{d}x\notag\\
        \leq&\int_{W_{\delta}^{x_{\epsilon}}\backslash\overline{W_{Rr_{\epsilon}^{N/(N-\beta)}}^{x_{\epsilon}}}}F(\nabla u_{\epsilon})^{N}\;\mathrm{d}x\notag\\
        =&1+\gamma\|u_{\epsilon}\|_{p}^{N}-\int_{\Omega\backslash W_{\delta}^{x_{\epsilon}}}F(\nabla u_{\epsilon})^{N}\;\mathrm{d}x-\int_{W_{Rr_{\epsilon}^{N/(N-\beta)}}^{x_{\epsilon}}}F(\nabla u_{\epsilon})^{N}\;\mathrm{d}x.
    \end{align}
    We estimate the two integrals on the right hand respectively. By Lemma \ref{le2.1} and Lemma \ref{le4.5},
    \begin{align}\label{eq4.24}
        &\int_{\Omega\backslash W_{\delta}^{x_{\epsilon}}}F(\nabla u_{\epsilon})^{N}\;\mathrm{d}x\notag\\
        =&c_{\epsilon}^{-\frac{N}{N-1}}\left(\int_{\Omega\backslash W_{\delta}^{x_{\epsilon}}}F(\nabla G)^{N}\;\mathrm{d}x+o_{\epsilon}(\delta)\right)\notag\\
        =&c_{\epsilon}^{-\frac{N}{N-1}}\left(-\int_{\Omega\backslash W_{\delta}^{x_{\epsilon}}}\mathrm{div}(F(\nabla G)^{N-1}\nabla F(\nabla G))G\;\mathrm{d}x\right.\notag\\
        &\left.\;\;\;\;\;\;\;\;\;\;\;\;+\int_{\partial(\Omega\backslash W_{\delta}^{x_{\epsilon}})}GF(\nabla G)^{N-1}\nabla F(\nabla G)\cdot v\;\mathrm{d}\sigma+o_{\epsilon}(\delta)\right)\notag\\
        =&c_{\epsilon}^{-\frac{N}{N-1}}\left(\gamma\|G\|_{p}^{N}-\int_{\partial W_{\delta}^{x_{\epsilon}}}GF(\nabla G)^{N-1}\nabla F(\nabla G)\cdot v\;\mathrm{d}\sigma+o_{\delta}(1)+o_{\epsilon}(\delta)\right)\notag\\
        =&c_{\epsilon}^{-\frac{N}{N-1}}\left(\gamma\|G\|_{p}^{N}-\frac{N}{\lambda_{N}}\log\delta+A_{0}+o_{\delta}(1)+o_{\epsilon}(\delta)\right)
    \end{align}
    and
    \begin{align}\label{eq4.25}
        \|u_{\epsilon}\|_{p}^{N}=c_{\epsilon}^{-\frac{N}{N-1}}(\|G\|_{p}^{N}+o_{\epsilon}(1)).
    \end{align}
    Notice that for any $T>0$,
    \begin{align}
        I_{N}(T):&=\int_{0}^{T}\frac{t^{N-1}}{(1+t)^{N}}\;\mathrm{d}t\notag\\
        &=\frac{1}{1-N}\int_{0}^{T}t^{N-1}\;\mathrm{d}(1+t)^{1-N}\notag\\
        &=\frac{1}{1-N}(\frac{T}{1+T})^{N-1}+I_{N-1}(T).\notag
    \end{align}
    Since $I_{1}(T)=\log(1+T)$, by iteration, we have
    $$
    I_{N}(T)=\log(1+T)-\sum_{k=1}^{N-1}\frac{1}{k}(\frac{T}{1+T})^{k}.
    $$
    By Lemma \ref{le2.1}, Lemma \ref{le4.2} and the change in variable $t=(\frac{\kappa_{N}}{1-\frac{\beta}{N}})^{\frac{1}{N-1}}r^{\frac{N-\beta}{N-1}}$,
    \begin{align}\label{eq4.26}
        &\int_{W_{Rr_{\epsilon}^{N/(N-\beta)}}^{x_{\epsilon}}}F(\nabla u_{\epsilon})^{N}\;\mathrm{d}x\notag\\
        =&\frac{1}{c_{\epsilon}^{\frac{1}{N-1}}}\left(\int_{W_{R}}F(\nabla\Phi)^{N}\;\mathrm{d}x+o_{\epsilon}(R^{-1})\right)\notag\\
        =&(\frac{1}{c_{\epsilon}^{\frac{1}{N-1}}}\frac{N}{\lambda_{N}})^{N}\left(\int_{W_{R}}(\frac{F^{o}(x)^{\frac{N-\beta}{N-1}-1}}{(\frac{\kappa_{N}}{1-\frac{\beta}{N}})^{-\frac{1}{N-1}}+F^{o}(x)^{\frac{N-\beta}{N-1}}})^{N}\;\mathrm{d}x+o_{\epsilon}(R^{-1})\right)\notag\\
        =&(\frac{1}{c_{\epsilon}^{\frac{1}{N-1}}}\frac{N}{\lambda_{N}})^{N}\left(\int_{0}^{R}(\frac{r^{\frac{N-\beta}{N-1}-1}}{(\frac{\kappa_{N}}{1-\frac{\beta}{N}})^{-\frac{1}{N-1}}+r^{\frac{N-\beta}{N-1}}})^{N}\int_{\partial W_{r}}\frac{1}{\lvert\nabla F^{o}(x)\rvert}\;\mathrm{d}\sigma\;\mathrm{d}r+o_{\epsilon}(R^{-1})\right)\notag\\
        =&\frac{1}{c_{\epsilon}^{\frac{N}{N-1}}}\frac{N}{\lambda_{N}}\left(\int_{0}^{R}\frac{r^{\frac{N-\beta}{N-1}N-1}}{((\frac{\kappa_{N}}{1-\frac{\beta}{N}})^{-\frac{1}{N-1}}+r^{\frac{N-\beta}{N-1}})^{N}}\;\mathrm{d}r+o_{\epsilon}(R^{-1})\right)\notag\\
        =&\frac{1}{c_{\epsilon}^{\frac{N}{N-1}}}\frac{N-1}{\lambda_{N}(1-\frac{\beta}{N})}\left(\int_{0}^{(\frac{\kappa_{N}}{1-\frac{\beta}{N}})^{\frac{1}{N-1}}R^{\frac{N-\beta}{N-1}}}\frac{t^{N-1}}{(1+t)^{N}}\;\mathrm{d}t+o_{\epsilon}(R^{-1})\right)\notag\\
        =&\frac{1}{c_{\epsilon}^{\frac{N}{N-1}}}\frac{N-1}{\lambda_{N}(1-\frac{\beta}{N})}\left(\log\left(1+(\frac{\kappa_{N}}{1-\frac{\beta}{N}})^{\frac{1}{N-1}}R^{\frac{N-\beta}{N-1}}\right)-\sum_{k=1}^{N-1}\frac{1}{k}(\frac{(\frac{\kappa_{N}}{1-\frac{\beta}{N}})^{\frac{1}{N-1}}R^{\frac{N-\beta}{N-1}}}{1+(\frac{\kappa_{N}}{1-\frac{\beta}{N}})^{\frac{1}{N-1}}R^{\frac{N-\beta}{N-1}}})^{k}+o_{\epsilon}(R^{-1})\right)\notag\\
        =&\frac{1}{c_{\epsilon}^{\frac{N}{N-1}}}\left(\frac{N}{\lambda_{N}}\log R+\frac{1}{\lambda_{N}(1-\frac{\beta}{N})}\log\frac{\kappa_{N}}{1-\frac{\beta}{N}}-\frac{N-1}{\lambda_{N}(1-\frac{\beta}{N})}\sum_{k=1}^{N-1}\frac{1}{k}+O(\log R^{-1})+o_{\epsilon}(R^{-1})\right)
    \end{align}
    Hence, by \eqref{eq4.23}-\eqref{eq4.26} and the inequality $(1+t)^{m}\leq 1+mt$ for $t\geq-1$ and $m\in[0,1]$,
    \begin{align}\label{eq4.27}
        &\left(\int_{W_{\delta}^{x_{\epsilon}}\backslash\overline{W_{Rr_{\epsilon}^{N/(N-\beta)}}^{x_{\epsilon}}}}F(\nabla h)^{N}\;\mathrm{d}x\right)^{\frac{1}{N-1}}\notag\\
        \leq&\left(1-c_{\epsilon}^{-\frac{N}{N-1}}\left(\frac{N}{\lambda_{N}}\log\frac{R}{\delta}+\frac{1}{\lambda_{N}(1-\frac{\beta}{N})}\log\frac{\kappa_{N}}{1-\frac{\beta}{N}}\right.\right.\notag\\
        &\left.\left.-\frac{N-1}{\lambda_{N}(1-\frac{\beta}{N})}\sum_{k=1}^{N-1}\frac{1}{k}+A_{0}+O(\log R^{-1})+o_{\epsilon}(1)+o_{\delta}(1)+o_{\epsilon}(\delta)+o_{\epsilon}(R^{-1})\right)\right)^{\frac{1}{N-1}}\notag\\
        \leq&1-\frac{c_{\epsilon}^{-\frac{N}{N-1}}}{N-1}\left(\frac{N}{\lambda_{N}}\log\frac{R}{\delta}+\frac{1}{\lambda_{N}(1-\frac{\beta}{N})}\log\frac{\kappa_{N}}{1-\frac{\beta}{N}}\right.\notag\\
        &\left.-\frac{N-1}{\lambda_{N}(1-\frac{\beta}{N})}\sum_{k=1}^{N-1}\frac{1}{k}+A_{0}+O(\log R^{-1})+o_{\epsilon}(1)+o_{\delta}(1)+o_{\epsilon}(\delta)+o_{\epsilon}(R^{-1})\right)\notag\\
        :=&1+b(\epsilon,\delta,R)\notag\\
        \rightarrow&1
   \end{align}
    as $\epsilon\rightarrow0$ for each fixed $\delta>0$ and $R>0$.

    Finally, combining \eqref{eq4.20}-\eqref{eq4.22} and \eqref{eq4.27} together, we have
    \begin{align}
        &(N\kappa_{N})^{\frac{1}{N-1}}\left(c_{\epsilon}^{\frac{N}{N-1}}+\frac{N}{N-1}\left(-\frac{N-1}{\lambda_{N}(1-\frac{\beta}{N})}\log(1+\frac{\lambda_{N}R^{\frac{N}{N-1}(1-\frac{\beta}{N})}}{N^{\frac{N}{N-1}}(1-\frac{\beta}{N})^{\frac{1}{N-1}}})\right.\right.\notag\\
        &\left.\left.+\frac{N}{\lambda_{N}}\log\delta-A_{0}+o_{\epsilon}(R^{-1})+o_{\delta}(1)+o_{\epsilon}(\delta)\right)\right)\notag\\
        \leq&\left(\log\frac{\delta}{R}-\frac{1}{N-\beta}\log(\lambda_{\epsilon}c_{\epsilon}^{-\frac{N}{N-1}})+\frac{\lambda_{N}(1-\frac{\beta}{N}-\epsilon)c_{\epsilon}^{\frac{N}{N-1}}}{N-\beta}\right)\notag\\
        &\cdot\left(1-\frac{c_{\epsilon}^{-\frac{N}{N-1}}}{N-1}\left(\frac{N}{\lambda_{N}}\log\frac{R}{\delta}+\frac{1}{\lambda_{N}(1-\frac{\beta}{N})}\log\frac{\kappa_{N}}{1-\frac{\beta}{N}}\right.\right.\notag\\
        &\left.\left.-\frac{N-1}{\lambda_{N}(1-\frac{\beta}{N})}\sum_{k=1}^{N-1}\frac{1}{k}+A_{0}+O(\log R^{-1})+o_{\epsilon}(1)+o_{\delta}(1)+o_{\epsilon}(\delta)+o_{\epsilon}(R^{-1})\right)\right)\notag\\
        =&(1+b(\epsilon,\delta,R))\left(\log\frac{\delta}{R}-\frac{1}{N-\beta}\log(\lambda_{\epsilon}c_{\epsilon}^{-\frac{N}{N-1}})\right)+\frac{\lambda_{N}(1-\frac{\beta}{N}-\epsilon)c_{\epsilon}^{\frac{N}{N-1}}}{N-\beta}\notag\\
        &-\frac{\lambda_{N}(1-\frac{\beta}{N}-\epsilon)}{(N-\beta)(N-1)}\left(\frac{N}{\lambda_{N}}\log\frac{R}{\delta}+\frac{1}{\lambda_{N}(1-\frac{\beta}{N})}\log\frac{\kappa_{N}}{1-\frac{\beta}{N}}\right.\notag\\
        &\left.-\frac{N-1}{\lambda_{N}(1-\frac{\beta}{N})}\sum_{k=1}^{N-1}\frac{1}{k}+A_{0}+O(\log R^{-1})+o_{\epsilon}(1)+o_{\delta}(1)+o_{\epsilon}(\delta)+o_{\epsilon}(R^{-1})\right).\notag
    \end{align}
    Rearranging the terms and letting $\epsilon\rightarrow0$ first and then $\delta\rightarrow0$, $R\rightarrow+\infty$, we have
    $$
    \limsup_{\epsilon\rightarrow0}\log\frac{\lambda_{\epsilon}}{c_{\epsilon}^{N/(N-1)}}\leq\log\frac{N\kappa_{N}}{N-\beta}+\sum_{k=1}^{N-1}\frac{1}{k}+\lambda_{N}(1-\frac{\beta}{N})A_{0}.
    $$
    Therefore we conclude our desired inequality by Lemma \ref{le4.4}.
\end{proof}

\section{Proof of Theorem \ref{th1.1} and Theorem \ref{th1.4}}\label{section5}
\subsection{Proof of Theorem \ref{th1.1}}

In case $c_{\epsilon}$ is bounded, the equality \eqref{eq4.1} implies Theorem \ref{th1.1}. In case $c_{\epsilon}\rightarrow+\infty$, Theorem \ref{th1.1}  is followed from Lemma \ref{le4.6}.

\subsection{Proof of Theorem \ref{th1.4}}

In case $c_{\epsilon}$ is bounded, the equality \eqref{eq4.1} implies Theorem \ref{th1.4}. In case $c_{\epsilon}\rightarrow+\infty$, Lemma \ref{le4.6} holds. We shall construct a sequence of functions $\varphi_{\epsilon}\in W_{0}^{1,N}(\Omega)$ with $\|\varphi_{\epsilon}\|_{\gamma,p,F}=1$ such that
$$
\int_{\Omega}\frac{e^{\lambda_{N}(1-\frac{\beta}{N})\varphi_{\epsilon}^{\frac{N}{N-1}}}}{F^{o}(x)^{\beta}}\;\mathrm{d}x>\int_{\Omega}\frac{1}{F^{o}(x)^{\beta}}\;\mathrm{d}x+\frac{N\kappa_{N}}{N-\beta}e^{\lambda_{N}(1-\frac{\beta}{N})A_{0}+\sum_{k=1}^{N-1}\frac{1}{k}}.
$$
which is the contradiction with Lemma \ref{le4.6}. Therefore $c_{\epsilon}$ must be bounded and the proof of Theorem \ref{th1.4} is finished.

Define a sequence of functions on $\Omega$ by
\begin{align}
    \varphi_{\epsilon}(x)=
    \left\{\begin{array}{l}
    c+\frac{1}{c^{\frac{1}{N-1}}}\left(b-\frac{N-1}{\lambda_{N}(1-\frac{\beta}{N})}\log\left(1+(\frac{\kappa_{N}}{1-\frac{\beta}{N}})^{\frac{1}{N-1}}(\frac{F^{o}(x)}{\epsilon})^{\frac{N-\beta}{N-1}}\right)\right)\;\mathrm{in}\;\overline{W_{R\epsilon}},\\
    \frac{G-\eta\xi}{c^{\frac{1}{N-1}}}\;\;\;\;\;\;\;\;\;\;\mathrm{in}\;W_{2R\epsilon}\backslash\overline{W_{R\epsilon}},\\
    \frac{G}{c^{\frac{1}{N-1}}}\;\;\;\;\;\;\;\;\;\;\mathrm{in}\;\Omega\backslash W_{2R\epsilon},\end{array}\right.\notag
\end{align}
    where $G$ and $\xi$ are functions given in Lemma \ref{le4.5}, $R=(-\log\epsilon)^{1/(1-\frac{\beta}{N})}$, $b$ and $c$ are constants depending only on $\epsilon$ to be determined later, $\eta$ is a cut-off function
    in $C_{0}^{\infty}(W_{2R\epsilon})$ with $\eta=1$ on $W_{R\epsilon}$ and $F(\nabla\eta)\leq\frac{2}{R\epsilon}$. Clearly $W_{2R\epsilon}\subset\Omega$ provided  $\epsilon$ is sufficiently small.

    In order to assure  $\varphi_{\epsilon}\in W_{0}^{1,N}(\Omega)$, we set
    $$
    c+\frac{1}{c^{\frac{1}{N-1}}}\left(b-\frac{N-1}{\lambda_{N}(1-\frac{\beta}{N})}\log(1+(\frac{\kappa_{N}}{1-\frac{\beta}{N}})^{\frac{1}{N-1}}R^{\frac{N-\beta}{N-1}})\right)=\frac{1}{c^{\frac{1}{N-1}}}(-\frac{N}{\lambda_{N}}\log(R\epsilon)+A_{0}),
    $$
    which implies
    \begin{align}\label{eq5.1}
        c^{\frac{N}{N-1}}=-\frac{N}{\lambda_{N}}\log\epsilon+A_{0}-b+\frac{1}{\lambda_{N}(1-\frac{\beta}{N})}\log\frac{\kappa_{N}}{1-\frac{\beta}{N}}+O(\log R^{-1}).
    \end{align}
    and $\varphi_{\epsilon}$ is continuous at any point $x\in\partial W_{R\epsilon}$.

    Since $\xi\in C^{1}(\overline{\Omega})$ and $\xi(x)=O(F^{o}(x))$ as $x\rightarrow0$, we have $F(\nabla(\eta\xi))=O(1)$ as $\epsilon\rightarrow0$. Moreover, $F(\nabla G)=O(F^{o}(x)^{-1})$ as $x\rightarrow0$. Hence,
    \begin{align}\label{eq5.2}
        &\int_{\Omega\backslash W_{R\epsilon}}F(\nabla\varphi_{\epsilon})^{N}\;\mathrm{d}x\notag\\
        =&\frac{1}{c^{\frac{N}{N-1}}}\left(\int_{\Omega\backslash W_{2R\epsilon}}F(\nabla G)^{N}\;\mathrm{d}x+\int_{W_{2R\epsilon}\backslash W_{R\epsilon}}F(\nabla(G-\eta\xi))^{N}\;\mathrm{d}x\right)\notag\\
        =&\frac{1}{c^{\frac{N}{N-1}}}\left(\int_{\Omega\backslash W_{2R\epsilon}}F(\nabla G)^{N}\;\mathrm{d}x+\int_{W_{2R\epsilon}\backslash W_{R\epsilon}}F(\nabla G)^{N}(1+O(F^{o}(x)^{N}))\;\mathrm{d}x\right)\notag\\
        =&\frac{1}{c^{\frac{N}{N-1}}}\left(\int_{\Omega\backslash W_{R\epsilon}}F(\nabla G)^{N}\;\mathrm{d}x+\int_{W_{2R\epsilon}\backslash W_{R\epsilon}}O(1)\;\mathrm{d}x\right)\notag\\
        =&\frac{1}{c^{\frac{N}{N-1}}}\left(\gamma\|G\|_{p}^{N}-\frac{N}{\lambda_{N}}\log(R\epsilon)+A_{0}+O(R\epsilon)+O((R\epsilon)^{N})\right)
    \end{align}
    and
    \begin{align}\label{eq5.3}
        &\int_{\Omega\backslash W_{R\epsilon}}\lvert\varphi_{\epsilon}\rvert^{p}\;\mathrm{d}x\notag\\
        =&\frac{1}{c^{\frac{N}{N-1}}}\left(\int_{\Omega\backslash W_{2R\epsilon}}\lvert G\rvert^{p}\;\mathrm{d}x+\int_{W_{2R\epsilon\backslash W_{R\epsilon}}}\lvert G-\eta\xi\rvert^{p}\;\mathrm{d}x\right)\notag\\
        =&\frac{1}{c^{\frac{N}{N-1}}}\left(\int_{\Omega\backslash W_{2R\epsilon}}\lvert G\rvert^{p}\;\mathrm{d}x+\int_{W_{2R\epsilon\backslash W_{R\epsilon}}}\lvert G\rvert^{p}(1+O((\frac{F^{o}(x)^{}}{F^{o}(x)+\log(F^{o}(x)^{-1})}))^{p})\;\mathrm{d}x\right)\notag\\
        =&\frac{1}{c^{\frac{N}{N-1}}}\left(\int_{\Omega\backslash W_{R\epsilon}}\lvert G\rvert^{p}\;\mathrm{d}x+\int_{W_{2R\epsilon\backslash W_{R\epsilon}}}O(F^{o}(x)^{p})\;\mathrm{d}x\right)\notag\\
        =&\frac{1}{c^{\frac{N}{N-1}}}\left(\|G\|_{p}^{p}+O(R\epsilon)+O((R\epsilon)^{N+p})\right).
    \end{align}
    By Lemma \ref{le2.1} and the change in variable $t=(\frac{\kappa_{N}}{1-\frac{\beta}{N}})^{\frac{1}{N-1}}(\frac{r}{\epsilon})^{\frac{N-\beta}{N-1}}$,
    \begin{align}\label{eq5.4}
        &\int_{W_{R\epsilon}}F(\nabla\varphi_{\epsilon})^{N}\;\mathrm{d}x\notag\\
        =&(\frac{1}{c^{\frac{1}{N-1}}}\frac{N}{\lambda_{N}})^{N}\int_{W_{R\epsilon}}(\frac{F^{o}(x)^{\frac{N-\beta}{N-1}-1}}{\epsilon^{\frac{N-\beta}{N-1}}(\frac{\kappa_{N}}{1-\frac{\beta}{N}})^{-\frac{1}{N-1}}+F^{o}(x)^{\frac{N-\beta}{N-1}}})^{N}\;\mathrm{d}x\notag\\
        =&(\frac{1}{c^{\frac{1}{N-1}}}\frac{N}{\lambda_{N}})^{N}\int_{0}^{R\epsilon}(\frac{r^{\frac{N-\beta}{N-1}-1}}{\epsilon^{\frac{N-\beta}{N-1}}(\frac{\kappa_{N}}{1-\frac{\beta}{N}})^{-\frac{1}{N-1}}+r^{\frac{N-\beta}{N-1}}})^{N}\int_{\partial W_{r}}\frac{1}{\lvert\nabla F^{o}(x)\rvert}\;\mathrm{d}\sigma\;\mathrm{d}r\notag\\
        =&\frac{1}{c^{\frac{N}{N-1}}}\frac{N}{\lambda_{N}}\int_{0}^{R\epsilon}\frac{r^{\frac{N-\beta}{N-1}N-1}}{\left(\epsilon^{\frac{N-\beta}{N-1}}(\frac{\kappa_{N}}{1-\frac{\beta}{N}})^{-\frac{1}{N-1}}+r^{\frac{N-\beta}{N-1}}\right)^{N}}\mathrm{d}r\notag\\
        =&\frac{1}{c^{\frac{N}{N-1}}}\frac{N-1}{\lambda_{N}(1-\frac{\beta}{N})}\int_{0}^{(\frac{\kappa_{N}}{1-\frac{\beta}{N}})^{\frac{1}{N-1}}R^{\frac{N-\beta}{N-1}}}\frac{t^{N-1}}{(1+t)^{N}}\;\mathrm{d}t\notag\\
        =&\frac{1}{c^{\frac{N}{N-1}}}\frac{N-1}{\lambda_{N}(1-\frac{\beta}{N})}\left(\log\left(1+(\frac{\kappa_{N}}{1-\frac{\beta}{N}})^{\frac{1}{N-1}}R^{\frac{N-\beta}{N-1}}\right)-\sum_{k=1}^{N-1}\frac{1}{k}(\frac{(\frac{\kappa_{N}}{1-\frac{\beta}{N}})^{\frac{1}{N-1}}R^{\frac{N-\beta}{N-1}}}{1+(\frac{\kappa_{N}}{1-\frac{\beta}{N}})^{\frac{1}{N-1}}R^{\frac{N-\beta}{N-1}}})^{k}\right)\notag\\
        =&\frac{1}{c^{\frac{N}{N-1}}}\left(\frac{N}{\lambda_{N}}\log R+\frac{1}{\lambda_{N}(1-\frac{\beta}{N})}\log\frac{\kappa_{N}}{1-\frac{\beta}{N}}-\frac{N-1}{\lambda_{N}(1-\frac{\beta}{N})}\sum_{k=1}^{N-1}\frac{1}{k}+O(\log R^{-1})\right).
    \end{align}
    Meanwhile, by \eqref{eq5.1},
    \begin{align}
        0\leq\varphi_{\epsilon}&\leq\frac{1}{c^{\frac{1}{N-1}}}\left(-\frac{N}{\lambda_{N}}\log R\epsilon+A_{0}+\frac{N-1}{\lambda_{N}(1-\frac{\beta}{N})}\log(1+(\frac{\kappa_{N}}{1-\frac{\beta}{N}})^{\frac{1}{N-1}}R^{\frac{N-\beta}{N-1}})\right)\notag\\
        &=O(\log(R\epsilon)^{-1})\;\mathrm{on}\;W_{R\epsilon}.\notag
    \end{align}
    So,
    \begin{align}\label{eq5.5}
        \int_{W_{R\epsilon}}\lvert\varphi_{\epsilon}\rvert^{p}\;\mathrm{d}x=O((R\epsilon)^{N}(\log(R\epsilon)^{-1})^{p})
    \end{align}
    Therefore, combining \eqref{eq5.2}-\eqref{eq5.5} together, we have
    \begin{align}\label{eq5.6}
        \|\varphi_{\epsilon}\|_{\gamma,p,F}=\frac{1}{c^{\frac{N}{N-1}}}&\left(-\frac{N}{\lambda_{N}}\log\epsilon+\frac{1}{\lambda_{N}(1-\frac{\beta}{N})}\log\frac{\kappa_{N}}{1-\frac{\beta}{N}}-\frac{N-1}{\lambda_{N}(1-\frac{\beta}{N})}\sum_{k=1}^{N-1}\frac{1}{k}+A_{0}+o_{\epsilon}(1)\right),
    \end{align}
    The term in bracket is positive for $\epsilon>0$ small enough. Hence, we can choose $c$ such that $\|\varphi_{\epsilon}\|_{\gamma,p,F}=1$, i.e.,
    \begin{align}\label{eq5.7}
        c^{\frac{N}{N-1}}=-\frac{N}{\lambda_{N}}\log\epsilon+\frac{1}{\lambda_{N}(1-\frac{\beta}{N})}\log\frac{\kappa_{N}}{1-\frac{\beta}{N}}-\frac{N-1}{\lambda_{N}(1-\frac{\beta}{N})}\sum_{k=1}^{N-1}\frac{1}{k}+A_{0}+o_{\epsilon}(1).
    \end{align}
    Combining \eqref{eq5.1} and \eqref{eq5.7} together, we have
    \begin{align}\label{eq5.8}
        b=\frac{N-1}{\lambda_{N}(1-\frac{\beta}{N})}\sum_{k=1}^{N-1}\frac{1}{k}+o_{\epsilon}(1).
    \end{align}
    Now we estimate $\int_{\Omega}\frac{e^{\lambda_{N}(1-\frac{\beta}{N})\varphi_{\epsilon}^{N/(N-1)}}}{F^{o}(x)^{\beta}}\;\mathrm{d}x$. By \eqref{eq5.7} and $F^{o}(x)<R\epsilon$ on $W_{R\epsilon}$,
    \begin{align}
        \frac{b-\frac{N-1}{\lambda_{N}(1-\frac{\beta}{N})}\log\left(1+(\frac{\kappa_{N}}{1-\frac{\beta}{N}})^{\frac{1}{N-1}}(\frac{F^{o}(x)}{\epsilon})^{\frac{N-\beta}{N-1}}\right)}{c^{\frac{N}{N-1}}}>-1\notag
    \end{align}
    for $\epsilon>0$ small enough. Then by \eqref{eq5.7}, \eqref{eq5.8} and the inequality $(1+t)^{m}\geq1+mt$ for $t\geq-1$ and $m\geq1$,
    \begin{align}\label{eq5.9}
        \varphi_{\epsilon}^{\frac{N}{N-1}}\geq& c^{\frac{N}{N-1}}+\frac{N}{N-1}\left(b-\frac{N-1}{\lambda_{N}(1-\frac{\beta}{N})}\log\left(1+(\frac{\kappa_{N}}{1-\frac{\beta}{N}})^{\frac{1}{N-1}}(\frac{F^{o}(x)}{\epsilon})^{\frac{N-\beta}{N-1}}\right)\right)\notag\\
        =&-\frac{N}{\lambda_{N}}\log\epsilon+\frac{1}{\lambda_{N}(1-\frac{\beta}{N})}\log\frac{\kappa_{N}}{1-\frac{\beta}{N}}+\frac{1}{\lambda_{N}(1-\frac{\beta}{N})}\sum_{k=1}^{N-1}\frac{1}{k}+A_{0}\notag\\
        &-\frac{N}{\lambda_{N}(1-\frac{\beta}{N})}\log\left(1+(\frac{\kappa_{N}}{1-\frac{\beta}{N}})^{\frac{1}{N-1}}(\frac{F^{o}(x)}{\epsilon})^{\frac{N-\beta}{N-1}}\right)+o_{\epsilon}(1)
    \end{align}
    on $W_{R\epsilon}$. Notice that by Lemma \ref{le2.1} and the change in variable $t=(\frac{\kappa_{N}}{1-\frac{\beta}{N}})^{\frac{1}{N-1}}(\frac{r}{\epsilon})^{\frac{N-\beta}{N-1}}$,
    \begin{align}\label{eq5.10}
        &\int_{W_{R\epsilon}}\frac{1}{(1+(\frac{\kappa_{N}}{1-\frac{\beta}{N}})^{\frac{1}{N-1}}(\frac{F^{o}(x)}{\epsilon})^{\frac{N-\beta}{N-1}})^{N}F^{o}(x)^{\beta}}\;\mathrm{d}x\notag\\
        =&\int_{0}^{R\epsilon}\frac{1}{(1+(\frac{\kappa_{N}}{1-\frac{\beta}{N}})^{\frac{1}{N-1}}(\frac{r}{\epsilon})^{\frac{N-\beta}{N-1}})^{N}r^{\beta}}\int_{\partial W_{r}}\frac{1}{\lvert\nabla F^{o}(x)\rvert}\;\mathrm{d}\sigma\;\mathrm{d}r\notag\\
        =&(N-1)\epsilon^{N-\beta}\int_{0}^{(\frac{\kappa_{N}}{1-\frac{\beta}{N}})^{\frac{1}{N-1}}R^{\frac{N-\beta}{N-1}}}\frac{t^{N-2}}{(1+t)^{N}}\;\mathrm{d}t\notag\\
        =&\epsilon^{N-\beta}\left(\int_{0}^{(\frac{\kappa_{N}}{1-\frac{\beta}{N}})^{\frac{1}{N-1}}R^{\frac{N-\beta}{N-1}}}\frac{(N-2)t^{N-3}}{(1+t)^{N-1}}\;\mathrm{d}t-\left.\frac{t^{N-2}}{(1+t)^{N-1}}\right\rvert_{0}^{(\frac{\kappa_{N}}{1-\frac{\beta}{N}})^{\frac{1}{N-1}}R^{\frac{N-\beta}{N-1}}}\right)\notag\\
        =&\epsilon^{N-\beta}\int_{0}^{(\frac{\kappa_{N}}{1-\frac{\beta}{N}})^{\frac{1}{N-1}}R^{\frac{N-\beta}{N-1}}}\frac{1}{(1+t)^{2}}\;\mathrm{d}t+O(R^{-\frac{N-\beta}{N-1}})\notag\\
        =&\epsilon^{N-\beta}+O(R^{-\frac{N-\beta}{N-1}}).
    \end{align}
    Combining \eqref{eq5.9} and \eqref{eq5.10} together, we have
    \begin{align}\label{eq5.11}
        \int_{W_{R\epsilon}}\frac{e^{\lambda_{N}(1-\frac{\beta}{N})\varphi_{\epsilon}^{\frac{N}{N-1}}}}{F^{o}(x)^{\beta}}\;\mathrm{d}x\geq&\frac{\kappa_{N}}{(1-\frac{\beta}{N})\epsilon^{N-\beta}}e^{\sum_{k=1}^{N-1}\frac{1}{k}+\lambda_{N}(1-\frac{\beta}{N})A_{0}+o_{\epsilon}(1)}\notag\\
        &\cdot\int_{W_{R\epsilon}}\frac{1}{(1+(\frac{\kappa_{N}}{1-\frac{\beta}{N}})^{\frac{1}{N-1}}(\frac{F^{o}(x)}{\epsilon})^{\frac{N-\beta}{N-1}})^{N}F^{o}(x)^{\beta}}\;\mathrm{d}x\notag\\
        &=\frac{\kappa_{N}}{(1-\frac{\beta}{N})}e^{\sum_{k=1}^{N-1}\frac{1}{k}+\lambda_{N}(1-\frac{\beta}{N})A_{0}+o_{\epsilon}(1)}+O(R^{-\frac{N-\beta}{N-1}}).
    \end{align}
    On the other hand, by \eqref{eq5.7} and the inequality $e^{t}\geq1+\frac{t^{N-1}}{(N-1)!}$,
    \begin{align}\label{eq5.12}
        &\int_{\Omega\backslash W_{R\epsilon}}\frac{e^{\lambda_{N}(1-\frac{\beta}{N})\varphi_{\epsilon}^{\frac{N}{N-1}}}}{F^{o}(x)^{\beta}}\;\mathrm{d}x\notag\\
        \geq&\int_{\Omega\backslash W_{2R\epsilon}}\frac{1}{F^{o}(x)^{\beta}}\;\mathrm{d}x+\frac{\lambda_{N}^{N-1}(1-\frac{\beta}{N})^{N-1}}{(N-1)!c^{\frac{N}{N-1}}}\int_{\Omega\backslash W_{2R\epsilon}}\frac{G^{N}}{F^{o}(x)^{\beta}}\;\mathrm{d}x\notag\\
        =&\int_{\Omega}\frac{1}{F^{o}(x)^{\beta}}\;\mathrm{d}x+\frac{\lambda_{N}^{N-1}(1-\frac{\beta}{N})^{N-1}}{(N-1)!c^{\frac{N}{N-1}}}\int_{\Omega}\frac{G^{N}}{F^{o}(x)^{\beta}}\;\mathrm{d}x+O((R\epsilon)^{N-\beta}).
    \end{align}
    Therefore, by \eqref{eq5.11} and \eqref{eq5.12}, we conclude
    $$
    \int_{\Omega}\frac{e^{\lambda_{N}(1-\frac{\beta}{N})\varphi_{\epsilon}^{\frac{N}{N-1}}}}{F^{o}(x)^{\beta}}\;\mathrm{d}x>\int_{\Omega}\frac{1}{F^{o}(x)^{\beta}}\;\mathrm{d}x+\frac{N\kappa_{N}}{N-\beta}e^{\lambda_{N}(1-\frac{\beta}{N})A_{0}+\sum_{k=1}^{N-1}\frac{1}{k}}
    $$
    for $\epsilon>0$ sufficiently small.

\def\refname{References }

\end{document}